\documentclass{elsarticle}

\usepackage{amsmath,amssymb,amsthm} 
\usepackage{hyperref}  
\usepackage{algpseudocode,algorithm}

\usepackage{todonotes} 
\usepackage{lineno} 

\newcommand{\norm}{\operatorname{Norm}}
\newcommand{\disc}{\operatorname{Disc}}
\newcommand{\diff}{\operatorname{Diff}}
\newcommand{\tr}{\operatorname{Trace}}
\newcommand{\gal}{\operatorname{Gal}}

\newcommand{\ord}{\operatorname{ord}}

\newcommand{\BB}[1]{\mathbb{#1}} 
\newcommand{\script}[1]{\mathcal{#1}} 
\newcommand{\PP}{\BB{P}}
\newcommand{\CC}{\BB{C}}
\newcommand{\RR}{\BB{R}}
\newcommand{\QQ}{\BB{Q}}
\newcommand{\ZZ}{\BB{Z}}
\newcommand{\FF}{\BB{F}}
\newcommand{\sO}{\script{O}}

\newcommand{\st}{\colon}

\theoremstyle{plain}
\newtheorem{theorem}{Theorem}[section]
\newtheorem{proposition}[theorem]{Proposition}
\newtheorem{lemma}[theorem]{Lemma}
\newtheorem{corollary}[theorem]{Corollary}

\theoremstyle{remark}
\newtheorem{remark}[theorem]{Remark}

\theoremstyle{definition}
\newtheorem{definition}[theorem]{Definition}
\newtheorem{example}[theorem]{Example}

\journal{arxiv}

\begin{document}

\begin{frontmatter}
	\title{Super-Isolated Abelian Varieties}
	\author{Travis Scholl}
	\address{
		Department of Mathematics
		\\ University of California, Irvine
		\\ {\tt schollt@uci.edu}
	}

\end{frontmatter}

\begin{abstract}
	We call an abelian variety over a finite field $\FF_q$ \emph{super-isolated} if its ($\FF_q$-rational) isogeny class contains a single isomorphism class. In this paper, we use the Honda-Tate theorem to characterize super-isolated ordinary simple abelian varieties by certain algebraic integers. Our main result is that for a fixed dimension $g \geq 3$, there are finitely many such varieties.
\end{abstract}

\section{Introduction}

The goal of this paper is to characterize abelian varieties $A$ defined over a finite field $\FF_q$ such that the $\FF_q$-isogeny class of $A$ contains a single $\FF_q$-isomorphism class. We call such varieties \emph{super-isolated} (see Definition~\ref{def:super-isolated-variety} below). They are a natural extension of isolated elliptic curves and abelian surfaces discussed in elliptic and hyperelliptic curve cryptography (ECC) \cite{koblitz2011elliptic,wenhan2012isolated,scholl2017isolated}. Our main result is that for $g \geq 3$, there are only finitely many super-isolated ordinary simple abelian varieties of dimension $g$ (see Corollary~\ref{cor:finite-ord-av-super-isol}). The proof is an application of Honda-Tate theory, and relies on counting certain Weil numbers. Also, it seems unlikely that any exist of cryptographic size (e.g. over a field with $\approx 2^{256}$ elements). This aspect of the results is particularly surprising because super-isolated varieties of cryptographic size are known to exist in the practical range (dimension $\leq 2$), as shown in the author's previous work on super-isolated surfaces \cite{scholl2018super}.

The original motivation for studying isolated curves comes from ECC, see \cite[Sec.~11]{koblitz2011elliptic}. The security of ECC depends on the difficulty of the elliptic curve discrete log problem (ECDLP). The ECDLP can be transferred between elliptic curves via isogenies. If an attacker can quickly compute an isogeny $E \to E'$ and solve the ECDLP on $E'$, then they can also quickly solve the ECDLP on $E$. The more isogenies $E$ admits, the more options an attacker has. This idea was used to construct an efficient attack on the ECDLP on a significant proportion of curves over certain extension fields \cite{menezes2006cryptographic}. Super-isolated curves do not admit isogenies, so an attacker cannot transfer the ECDLP away from a super-isolated curve. A super-isolated curve may not be super-isolated after passing to a larger base field, but if the attacker requires a certain base field (as in \cite{menezes2006cryptographic}) then extending the base field is not advantageous. Even though abelian varieties of dimension $\geq 3$ are rarely used in cryptography, it is still mathematically interesting to ask whether higher dimensional super-isolated varieties exist.

This paper, which is a self-contained version of a chapter from the author's PhD thesis \cite{scholl2018abelian}, is organized as follows. In Section~\ref{sec:background} we review some standard results in algebraic number theory that will be used in the following sections. In Section~\ref{sec:weil-gens} we introduce certain algebraic integers called \emph{Weil generators}, which will later be used to characterize super-isolated varieties. In Section~\ref{sec:algorithm}, we outline an algorithm to enumerate Weil generators in a given number field. Our main result on Weil generators is Theorem~\ref{thm:W-size} in Section~\ref{sec:counting-weil-gens}. This result can be made effective for $g = 3$, and we give detailed examples computing an explicit bound on the number of Weil generators in a given field in Section~\ref{sec:effectiveness}. In Section~\ref{sec:super-isolated-variaties}, we apply the results on Weil generators to study super-isolated varieties.

\subsection*{Acknowledgments}

I would like to thank my advisor Neal Koblitz for all of his great support.

\section{Background}\label{sec:background}

The goal of this section is to recall some standard facts from algebraic number theory and to set notation.

For an extension of number fields $K/F$, we denote the relative discriminant and relative different ideals by $\disc_{K/F}$ and $\diff_{K/F}$ respectively. If the field $F$ is not given, then it is assumed to be $\QQ$. We denote the class number of $K$ by $h_K$.

For any $\alpha \in K$, let $\disc_{K/F}(\alpha)$ and $\diff_{K/F}(\alpha)$ denote the discriminant and different of $\alpha$ respectively. If $f \in F[x]$ is the characteristic polynomial of $\alpha$ over $F$, then by definition $\disc_{K/F}(\alpha)$ is the discriminant of $f$ and $\diff_{K/F}(\alpha)$ is $f'(\alpha)$.
Note that if $K \neq F(\alpha)$, then $\disc_{K/F}(\alpha) = \diff_{K/F}(\alpha) = 0$.

\begin{example}
	Suppose $K$ is a complex multiplication (CM) field and $F$ is the maximal totally real subfield. For any $\alpha \in K$, the characteristic polynomial of $\alpha$ over $F$ is $f(x) = x^2 - (\alpha + \overline{\alpha})x + \alpha\overline{\alpha}$. So $\disc_{K/F}(\alpha) = (\alpha - \overline{\alpha})^2$ and $\diff_{K/F}(\alpha) = \alpha - \overline{\alpha}$.
\end{example}

\begin{lemma}\label{lem:disc-is-norm-diff}
	Let $K/F$ be an arbitrary extension of number fields and let $\alpha \in K$ such that $K = F(\alpha)$. Then
	\[
	\disc_{K/F}(\alpha) = (-1)^{\binom{[K:F]}{2}}\norm_{K/F}\left(\diff_{K/F}(\alpha)\right).
	\]
\end{lemma}
\begin{proof}
	The proof is the same as in the special case with $F = \QQ$, which appears in \cite[Thm.~8]{marcus2018number}.
	%
	%
\end{proof}

\begin{lemma}\label{lem:relative-monogenic-iff-relative-different-prinicpal}
	Let $K/F$ be an arbitrary extension of number fields. Then
	\[
		\sO_K = \sO_F[\alpha]
		\quad\Leftrightarrow\quad
		(\diff_{K/F}(\alpha)) = \diff_{K/F}.
	\]
\end{lemma}
\begin{proof}
	The forward direction is proved in \cite[Prop.~2.4, Pg.~197]{neukrich1999algebraic}, so it remains to prove the reverse direction. Suppose that $(\diff_{K/F}(\alpha)) = \diff_{K/F}$. Then by Lemma~\ref{lem:disc-is-norm-diff} and  \cite[Thm.~2.9, Pg.~201]{neukrich1999algebraic},
	\[
		\left(\disc_{K/F}(\alpha)\right)
		= \left(\norm_{K/F}\left(\diff_{K/F}(\alpha)\right)\right)
		= \norm_{K/F} \left(\diff_{K/F}\right)
		= \disc_{K/F}.
	\]
	Recall that if $\{\beta_j\}$ is a basis for $\sO_F$, then $\{\alpha^i\beta_j\}$ spans a $\ZZ$-submodule of $\sO_K$ and has discriminant
	\[
		\norm_{F/\QQ}\left(\disc_{K/F}(\alpha)\right)\disc_{F}^{[K:F]}
		=
		\norm_{F/\QQ}\left(\disc_{K/F}\right)\disc_F^{[K:F]},
	\]
	see \cite[Ch.~2, Exercise~23]{marcus2018number}.
	The right hand side is $\disc_K$ by \cite[Cor.~2.10, Pg.~202]{neukrich1999algebraic}. In particular, $\{\beta_j\alpha^i\}$ is a basis for $\sO_K$, so $\sO_K = \sO_F[\alpha]$.
\end{proof}

\begin{definition}\label{defn:height}
	Let $\alpha \in \overline{\QQ}$. The \emph{height} of $\alpha$ is
	\[
		h(\alpha) = \max_{\sigma:\QQ(\alpha) \to \CC} \left|\sigma(\alpha)\right|.
	\]
\end{definition}

\section{Weil Generators}\label{sec:weil-gens}

The goal of this section is to define certain algebraic integers we call \emph{Weil generators}, and to give some of their properties.

\begin{definition}\label{defn:weil-number}
	Let $K$ be a CM field. We say $\alpha \in \sO_K$ is a \emph{Weil number} if $\alpha\overline{\alpha} \in \ZZ$. If $\alpha\overline{\alpha} = n$, then we say $\alpha$ is a Weil $n$-number.
\end{definition}

\begin{remark}\label{rem:height-norm-equiv}
	If $\alpha$ is a Weil number in $K$, then $|\sigma(\alpha)|$ is constant for all $\sigma:K \to \CC$. This means that $|\norm_{K/\QQ}(\alpha)| = h(\alpha)^{[K:\QQ]}$.
\end{remark}

\begin{definition}\label{defn:weil-generator}
	Let $K$ be a CM field. We say $\alpha \in \sO_K$ is a \emph{Weil generator for $K$} if $\alpha$ is a Weil number and $\ZZ[\alpha,\overline{\alpha}] = \sO_K$.
\end{definition}

\begin{example}\label{ex:operations-on-weil-gens}
	If $\alpha$ is a Weil number in $K$, then so is $\zeta\alpha$ for any root of unity $\zeta$ in $K$. However, this does not hold for Weil generators. For example, if $K = \QQ(i)$, then $i$ is a Weil generator for $K$, but $i^2 = -1$ is not. However, if $\alpha$ is a Weil generator, then so are all of the conjugates of $\alpha$, as well as $-\alpha$.
\end{example}

\begin{example}\label{ex:siw-in-g=1}
	If $K$ is any quadratic imaginary field and $\sO_K = \ZZ[\gamma]$, then $\alpha$ is a Weil generator for $K$ if and only if $\alpha = a \pm \gamma$ for some $a \in \ZZ$.
\end{example}

\begin{remark}\label{rem:weil-gen-implies-principal-different}
	If $\alpha$ is a Weil generator for $K$ and $F$ is the maximal totally real subfield of $K$, then $\sO_K = \sO_F[\alpha]$. By Lemma~\ref{lem:relative-monogenic-iff-relative-different-prinicpal}, this implies that $(\diff_{K/F}(\alpha)) = \diff_{K/F}$.
\end{remark}

\begin{example}\label{ex:no-W-because-no-gamma}
	Let $K = \QQ(\sqrt{10},\sqrt{-13})$. Then $\diff_{K/F} = (26,13 + \sqrt{-13})$ is not a principal ideal. Therefore $K$ does not contain a Weil generator by Remark~\ref{rem:weil-gen-implies-principal-different}.
\end{example}

\begin{example}\label{ex:no-relative-generator}
	Let $K = \QQ(\sqrt{15},\sqrt{-2})$. We claim that there is no $\alpha$ such that $\sO_K = \sO_F[\alpha]$, which implies that $K$ does not contain a Weil generator. Suppose for contradiction that such an $\alpha$ exists. One can compute that $\diff_{K/F} = (2)$. Therefore $\alpha - \overline{\alpha} = 2u$ for some $u \in \sO_K^\times$. But one can show that $\sO_K^\times = \sO_F^\times$, so any such $u$ actually lies in $F$. This is a contradiction because conjugation negates $\alpha - \overline{\alpha}$ but fixes $2u$.
\end{example}

\begin{example}
	Let $\zeta_n$ be a primitive $n$th root of unity with $n \geq 3$, and let $K = \QQ(\zeta_n)$. Then $\zeta_n$ is a Weil generator for $K$ because $\sO_K = \ZZ[\zeta_n]$ and $\zeta_n\overline{\zeta}_n = 1$.
\end{example}

\begin{lemma}\label{lem:siw-implies-F-monogenic}
	Let $K$ be a CM field and $F$ be the maximal totally real subfield of $K$. If $\alpha$ is a Weil generator for $K$, then $\sO_F = \ZZ[\alpha + \overline{\alpha}]$.
\end{lemma}
\begin{proof}
	By hypothesis, every element of $\sO_K$ can be written as polynomial in $\alpha,\overline{\alpha}$ with coefficients in $\ZZ$. $\sO_F$ is precisely the polynomials which are symmetric in $\alpha$ and $\overline{\alpha}$. Recall that the subring of symmetric polynomials in $\ZZ[x,y]$ is $\ZZ[x+y,xy]$. Therefore, $\sO_F = \ZZ[\alpha+\overline{\alpha},\alpha\overline{\alpha}]$. But $\alpha\overline{\alpha} \in \ZZ$, so this is the same as $\ZZ[\alpha+\overline{\alpha}]$.
\end{proof}

\begin{remark}
	If $\alpha$ is a Weil number in $K$, then the property $\sO_F = \ZZ[\alpha + \overline{\alpha}]$ does not imply $\alpha$ is a Weil generator. For example, if $K$ is a quadratic imaginary field, then every $\alpha \in \sO_K$ satisfies $\alpha\overline{\alpha} \in \ZZ$. Moreover, in this case $F = \QQ$ so $\sO_F = \ZZ = \ZZ[\alpha+\overline{\alpha}]$. However, not every $\alpha \in \sO_K$ will satisfy $\sO_K = \ZZ[\alpha,\overline{\alpha}]$, e.g. $K = \QQ(i)$ and $\alpha = 2i$.
\end{remark}

\begin{example}\label{ex:2-splits-no-weil-gens}
	Let $K = \QQ[x]/(x^6 + 12x^4 + 17x^2 + 2)$. Then $K$ is a CM field of degree $6$.
	Moreover, the prime 2 splits completely in $F$, hence there are $3$ maps $\sO_F \to \FF_2$. This shows that $F$ is not monogenic as there are only 2 maps $\ZZ[x] \to \FF_2$. Therefore $K$ does not contain any Weil generators.
\end{example}


\begin{lemma}\label{lem:equiv-defn-for-siw}
	Let $K$ be a CM field and $F$ the maximal totally real subfield of $K$. Then $\alpha \in \sO_K$ is a Weil generator for $K$ if and only if the following hold
	\begin{enumerate}
		\item\label{defn:siw2-nrm} $\alpha\overline{\alpha} \in \ZZ$
		\item\label{defn:siw2-monogenic-over-F} $\ZZ[\alpha+\overline{\alpha}] = \sO_F$
		\item\label{defn:siw2-diff} $(\diff_{K/F}(\alpha)) = \diff_{K/F}$
	\end{enumerate}
\end{lemma}
\begin{proof}
	Suppose that $\alpha$ is a Weil generator for $K$. Then $\alpha$ satisfies property~\ref{defn:siw2-nrm} by definition, and property~\ref{defn:siw2-monogenic-over-F} follows from Lemma~\ref{lem:siw-implies-F-monogenic}. Note that $\sO_K = \ZZ[\alpha,\overline{\alpha}]$ implies that $\sO_K = \sO_F[\alpha]$. So by Lemma~\ref{lem:relative-monogenic-iff-relative-different-prinicpal}, $\alpha$ satisfies property~\ref{defn:siw2-diff}.

	Now suppose that $\alpha \in K$ satisfies properties~\ref{defn:siw2-nrm}-\ref{defn:siw2-diff}. By Lemma~\ref{lem:relative-monogenic-iff-relative-different-prinicpal} and property~\ref{defn:siw2-diff}, $\sO_K = \sO_F[\alpha]$ and $\sO_F = \ZZ[\alpha+\overline{\alpha}]$. Hence $\sO_K = \ZZ[\alpha,\alpha+\overline{\alpha}] = \ZZ[\alpha,\overline{\alpha}]$, so $\alpha$ is a Weil generator.
\end{proof}

\begin{remark}
	The properties in Lemma~\ref{lem:equiv-defn-for-siw} are independent as shown by the following examples:
	\begin{itemize}
		\item If $K = \QQ(i)$ and $\alpha = 2i$, then  \ref{defn:siw2-nrm} and \ref{defn:siw2-monogenic-over-F} hold, but not \ref{defn:siw2-diff}.
		\item If $K = \QQ(\zeta_5)$ and $\alpha = \zeta_5 + 1$, then \ref{defn:siw2-monogenic-over-F} and \ref{defn:siw2-diff} hold, but not \ref{defn:siw2-nrm}.
		\item If $K = \QQ(\zeta_5)$ and $\alpha = -5\zeta_5^3 - 4\zeta_5^2 + 2\zeta_5 - 2$,
		then \ref{defn:siw2-nrm} and \ref{defn:siw2-diff} hold, but not \ref{defn:siw2-monogenic-over-F}.
	\end{itemize}
\end{remark}

Next we will show that we can always write Weil generators in a certain form. To do this, we first introduce some notation. Let $K$ be a fixed CM field of degree $2g$. Let $F$ be the maximal totally real subfield of $K$. Fix $\gamma \in \sO_K$ such that $\sO_K = \sO_F[\gamma]$.\footnote{It is possible that such a $\gamma$ does not exist, as in Example~\ref{ex:no-relative-generator}. However, by Lemma~\ref{lem:relative-monogenic-iff-relative-different-prinicpal} and Lemma~\ref{lem:equiv-defn-for-siw}, if no such $\gamma$ exists, then $K$ does not contain a Weil generator.} Let $T$ denote a set of representatives of the set of $\eta \in \sO_F$ such that $\ZZ[\eta] = \sO_F$ modulo integer translation. That is, for every $\eta'$ such that $\ZZ[\eta'] = \sO_F$, there exists a unique $\eta \in T$ such that $\eta' - \eta \in \ZZ$.

\begin{lemma}\label{lem:existence-of-formula-for-siw}
	If $\alpha \in K$ is a Weil generator, then
	\begin{equation}\label{eq:formula-for-alpha-siw}
	\alpha = \frac{u(\gamma - \overline{\gamma}) + \eta + a}{2}
	\end{equation}
	for a unique $u \in \sO_F^\times$, $\eta \in T$, and $a \in \ZZ$.
\end{lemma}

\begin{proof}
	First we will show that given $\alpha$, there exists some choice of $u$, $\eta$, and $a$ satisfying equation~(\ref{eq:formula-for-alpha-siw}). Then we will show that such a triple is unique.

	Given $\alpha \in W$, set $u = (\alpha - \overline{\alpha})/(\gamma - \overline{\gamma})$. Note that $u$ is a unit in $\sO_F^\times$ because $(\gamma - \overline{\gamma})\sO_K = \diff_{K/F} = (\alpha - \overline{\alpha})\sO_K$ by Lemma~\ref{lem:relative-monogenic-iff-relative-different-prinicpal} and Lemma~\ref{lem:equiv-defn-for-siw}. Recall that $\ZZ[\alpha + \overline{\alpha}] = \sO_F$, so by the definition of $T$ there exists a unique $\eta \in T$ and $a \in \ZZ$ such that $\alpha + \overline{\alpha} = \eta + a$. A straightforward calculation shows that $u$, $\eta$, and $a$ satisfy equation~(\ref{eq:formula-for-alpha-siw}).

	The uniqueness of $u$, $\eta$, and $a$ follows from their construction in terms of $\alpha$ (using our fixed choices of $\gamma$ and $T$). For example, if $u'$, $\eta'$, and $a'$ was another triple satisfying equation~(\ref{eq:formula-for-alpha-siw}), then both $u$ and $u'$ must equal $(\alpha - \overline{\alpha})/(\gamma - \overline{\gamma})$. We also have $\alpha + \overline{\alpha} = \eta + a = \eta' + a'$, so $\eta = \eta'$ and $a = a'$ by the definition of $T$.
\end{proof}

Next we recall a theorem of Gy{\"o}ry which implies that the set $T$ is finite (see \cite{gyory2000discriminant} for an English summary). This means that the number of possible $\eta$ (up to integer translation) in equation~(\ref{eq:formula-for-alpha-siw}) is finite.
\begin{theorem}[\!\!{\cite{gyory1976polynomials}}]\label{thm:finite-monogenerators}
	For any number field $L$, the set $T$ of $\eta$ such that $\sO_L = \ZZ[\eta]$, up to integer translation, is finite. Moreover, representatives for $T$ can be effectively determined.
\end{theorem}

\begin{example}\label{ex:mono-gens-in-quadratic-field}
	If $\deg L = 2$, then we may choose $T = \left\{ (\disc_{L} \pm \sqrt{\disc_{L}})/2 \right\}$. Hence $T$ has cardinality $2$.
\end{example}

\begin{example}
	If $\deg L = 3$, then finding $\eta$ such that $\sO_L = \ZZ[\eta]$ can be reduced to solving a certain Thue equation \cite{gaal1989computing}.
\end{example}

Equation~(\ref{eq:formula-for-alpha-siw}) suggests that one way to search for Weil generators is to fix $\gamma$ and enumerate over values for $\eta$, $u$, and $a$. The following lemma gives an optimization: when $g \geq 2$, there is at most one possible value of $a$.

\begin{lemma}\label{lem:uniqueness-of-formula-for-siw}
	If $g \geq 2$, then for any $u \in \sO_F^\times$ and $\eta \in T$, there is at most one $a \in \ZZ$ such that the right hand side of equation~(\ref{eq:formula-for-alpha-siw}) is a Weil generator.
\end{lemma}
\begin{proof}
	Let $u \in \sO_F^\times$ and $\eta \in T$. Let $\Omega = (u(\gamma - \overline{\gamma}) + \eta)/2$. It is sufficient to show that there is at most one $a \in \QQ$ such that $\alpha = \Omega + a/2$ satisfies $\alpha\overline{\alpha} \in \QQ$. This is a necessary condition for $\alpha$ to be a Weil generator. A straightforward computation shows that $\alpha\overline{\alpha} \in \QQ$ if and only if $\Omega\overline{\Omega} + a\eta/2 \in \QQ$. Because $\{1,\eta,\dots,\eta^g\}$ is a $\QQ$-basis for $F$, we may write $\Omega\overline{\Omega} = \sum a_i\eta^i$ for unique rational numbers $a_0,\dots,a_g \in \QQ$. Then $\Omega\overline{\Omega} + a\eta/2 \in \QQ$ if and only if $a = -2a_1$ and $a_i = 0$ for all $i > 1$.
\end{proof}

\section{Searching For Weil Generators}\label{sec:algorithm}

The goal of this section is to describe an efficient method for searching for Weil generators in a given CM field $K$. We will use the same notation as in Section~\ref{sec:weil-gens}.

A naive approach to finding Weil generators is to directly search over all elements of $\sO_K$. Using Lemma~\ref{lem:equiv-defn-for-siw}, one can quickly test whether a given $\alpha \in \sO_K$ is a Weil generator. This approach is impractical because Weil generators are sparse, as shown in Theorem~\ref{thm:W-size} below. Instead, we will enumerate units in $F$ and use those to attempt to construct Weil generators.

Recall from Lemma~\ref{lem:existence-of-formula-for-siw} that every Weil generator $\alpha$ can be written as
\[
	\alpha = \frac{u(\gamma - \overline{\gamma}) + \eta + a}{2},
\]
for a unique $u \in \sO_F^\times$, $\eta \in T$, and $a \in \ZZ$. Moreover, by Lemma~\ref{lem:uniqueness-of-formula-for-siw}, $a$ is uniquely determined by $u$ and $\eta$. Therefore, by searching over all $u$ and $\eta$, we will eventually find all Weil generators $\alpha$. This approach is formalized into Algorithm~\ref{alg:find-siw-g>=2} below.

\begin{center}
	\begin{minipage}{0.91\linewidth}
		\begin{algorithm}[H]
			\caption{Search Weil Generators}\label{alg:find-siw-g>=2}
			\begin{algorithmic}[1]
				\Require A CM field $K$ of degree $2g$, with $g \geq 2$, and a bound $B$
				\Ensure A set of Weil generators in $K$.
				\State $F \gets $ the maximal totally real subfield of $K$
				\State $\gamma \gets $ an element of $K$ such that $\sO_K = \sO_F[\gamma]$
				\State $T \gets $ a complete set of $\eta \in F$ such that $\sO_F = \ZZ[\eta]$ up to integer translation
				\State $U \gets $ all units $u \in \sO_F^\times$ with height $h(u) \leq B$
				\ForAll{$u\in U$ and $\eta \in T$}
				\State $\Omega \gets (u(\gamma - \overline{\gamma}) + \eta)/2$
				\State Write $\Omega\overline{\Omega} = \sum_{i=0}^{g-1} a_i\eta^i$ with $a_0,\dots,a_{g-1} \in \QQ$.
				\State $\alpha \gets \Omega - a_1$
				\If{$a_i = 0$ for $i > 1$ and $\alpha \in \sO_K$}
				\State \textbf{print} $\alpha$
				\EndIf
				\EndFor
			\end{algorithmic}
		\end{algorithm}
	\end{minipage}
\end{center}

\begin{theorem}
	Every $\alpha$ output by the Algorithm~\ref{alg:find-siw-g>=2} is a Weil generator. Moreover, for every Weil generator $\alpha \in K$, if Algorithm~\ref{alg:find-siw-g>=2} is given a sufficiently large bound $B$, then it will eventually print $\alpha$.
\end{theorem}
\begin{proof}
	Suppose that the algorithm outputs $\alpha$. Then $\alpha = \Omega - a_1$ where $\Omega = (u(\gamma - \overline{\gamma}) + \eta)/2$ and $a_1 \in \QQ$. Because $\Omega\overline{\Omega} - a_1\eta \in \QQ$, it follows that $(\Omega - a_1)(\overline{\Omega} - a_1) = \alpha\overline{\alpha} \in \QQ$. Recall that $\alpha \in \sO_K$ by construction, so $\alpha\overline{\alpha} \in \ZZ$. This also shows that $a_1 \in \QQ \cap \frac{1}{2}\sO_K = \frac{1}{2}\ZZ$. Hence $\alpha + \overline{\alpha} = \eta - 2a_1 \in \eta + \ZZ$, so $\ZZ[\alpha+\overline{\alpha}] = \sO_F$. Also, $\alpha - \overline{\alpha} = \Omega - \overline{\Omega} = u(\gamma - \overline{\gamma})$ so $(\diff_{K/F}(\alpha)) = (\diff_{K/F}(\gamma)) = \diff_{K/F}$. By Lemma~\ref{lem:equiv-defn-for-siw}, this shows that $\alpha$ is a Weil generator.

	Now suppose that $\alpha$ is a Weil generator for $K$. We want to show that for a large enough bound $B$, Algorithm~\ref{alg:find-siw-g>=2} will eventually output $\alpha$. By Lemma~\ref{lem:existence-of-formula-for-siw}, $\alpha = (u(\gamma - \overline{\gamma}) + \eta + a)/2$ for unique $u \in \sO_F^\times$, $\eta \in T$, and $a \in \ZZ$. If $B \geq h(u)$, then Algorithm~\ref{alg:find-siw-g>=2} is guaranteed to find $\alpha$ because it enumerates all possible $\eta$ and $a$ (which corresponds to $-2a_1$ in the notation of Algorithm~\ref{alg:find-siw-g>=2}) such that $(u(\gamma - \overline{\gamma}) + \eta + a)/2$ is a Weil generator.
\end{proof}

\begin{remark}
	Using Lemma~\ref{lem:bound-alpha-by-u-squared}, one can show that if $\deg K = 4$, then Algorithm~\ref{alg:find-siw-g>=2} returns all Weil generators $\alpha$ with $h(\alpha) \leq C_5 B^2$ for a constant $C_5$ which can be explicitly computed.
\end{remark}

%

\section{Counting Weil Generators}\label{sec:counting-weil-gens}

In this section, we state and prove our main result on the number of Weil generators of bounded height in a given CM field $K$ of degree $2g$.

\begin{theorem}\label{thm:W-size}
	Let $W$ be the set of Weil generators in a CM field $K$ of degree $2g$. Then
	\[
		\#\left\{ \alpha \in W \st h(\alpha) \leq N \right\}
		=
		\begin{cases}
			4N + O(1) & g = 1 \\
			\rho\log N + O(1) & g = 2 \text{\ and \ } W \neq \emptyset \\
			O(1) & g \geq 3,
		\end{cases}
	\]
	where $\rho$ is a constant depending on $K$. Moreover, if $g \leq 3$ then the implicit constants are effectively computable.\footnote{For example, if $g = 1$, then there is an computable constant $C$ such that $\left|\#\left\{ \alpha \in W \st h(\alpha) \leq N \right\} - 4N \right| \leq C$. The value of $C$ will depend on the field $K$ and choice of $\gamma$ and $T$.}
\end{theorem}

To prove Theorem~\ref{thm:W-size}, we proceed by cases depending on the degree of $K$. The case $g = 1$ is given by Proposition~\ref{prop:W-size-g=1} in Section~\ref{sec:g=1}. The case $g = 2$ is given by Proposition~\ref{prop:W-size-g=2} in Section~\ref{sec:g=2}. The case $g \geq 3$ is given by Proposition~\ref{prop:W-size-g>=3} in Section~\ref{sec:g>=3}.

Throughout this section, we will keep the notation introduced at the end of Section~\ref{sec:weil-gens}. Unless otherwise noted,
\begin{itemize}
	\item $K$ is a fixed CM field of degree $2g$.
	\item $F$ is the maximal totally real subfield of $K$.
	\item $\gamma$ is a fixed element of $K$ such that $\sO_K = \sO_F[\gamma]$.
	\item $T$ is a set of representatives of generators for $\sO_F$ as monogenic order, up to integer translation.
\end{itemize}

\begin{remark}
	It is possible that such a $\gamma$ does not exist, as in Example~\ref{ex:no-relative-generator}. However, by Lemma~\ref{lem:relative-monogenic-iff-relative-different-prinicpal} and Lemma~\ref{lem:equiv-defn-for-siw}, a necessary condition for $K$ to contain a Weil generator is that such a $\gamma$ exists.
\end{remark}

\subsection{The Case $g = 1$}\label{sec:g=1}

The goal of this section is to prove Theorem~\ref{thm:W-size} in the case $g = 1$.

\begin{proposition}\label{prop:W-size-g=1}
	If $K$ is a quadratic imaginary field, then
	\[
		\#\left\{ \alpha \in W \st h(\alpha) \leq N \right\} = 4N + O(1).
	\]
\end{proposition}
\begin{proof}
	Let $-d = \disc_{K}$ and let $\omega = (d + \sqrt{-d})/2$. Then $\sO_K = \ZZ[\omega]$. Recall from Example~\ref{ex:siw-in-g=1} that $\alpha \in K$ is a Weil generator if and only if $\alpha = a \pm \omega$ for some $a \in \ZZ$. The claim follows because $h(a \pm \omega) = |a| + O(1)$.
\end{proof}

\subsection{The Case $g = 2$}\label{sec:g=2}

The goal of this section is to prove Theorem~\ref{thm:W-size} in the case $g = 2$.

\begin{proposition}\label{prop:W-size-g=2}
	Let $K$ be a quartic CM field. There is a constant $\rho$ that depends only on $K$ such that if $W \neq \emptyset$, then
	\[
		\#\left\{ \alpha \in W \st h(\alpha) \leq N \right\} = \rho\log N + O(1).
	\]
	Moreover, both $\rho$ and the implied constant in $O(1)$ are effectively computable.
\end{proposition}
The main idea behind the proof of Proposition~\ref{prop:W-size-g=2} is to show that counting Weil numbers reduces to counting solutions to Pell's equation.

We will use the same notation as given in Section~\ref{sec:counting-weil-gens}. Because $F$ is a real quadratic field, we can set $T = \{\pm(\disc_F + \sqrt{\disc_F})/2\}$. However, there is no obvious choice of $\gamma$ because $\sO_K$ is not always a free $\sO_F$-module, see Example~\ref{ex:no-W-because-no-gamma}. In this section, we continue to assume some such $\gamma$ exists.

Some of the implied constants in this section will depend on the choice of $\gamma$ and $T$ as will the implied constant in the proposition. But the constant $\rho$ in Proposition~\ref{prop:W-size-g=2} depends only on the field $K$. A detailed explanation of how to compute these constants is given in Example~\ref{ex:asymptotic-value-number-W-N-zeta5} below.

The outline of the proof of Proposition~\ref{prop:W-size-g=2} is as follows. Let
\begin{align*}
	\Omega(u,\eta) &= \frac{u(\gamma - \overline{\gamma}) + \eta}{2},
	\\
	a(u,\eta) &= \text{The unique element of $(1/2)\ZZ$ such that } \\&\quad \norm_{K/F}\left(\Omega(u,\eta) + \frac{a(u,\eta)}{2}\right) \in \QQ, \text{ see the proof of Lemma~\ref{lem:uniqueness-of-formula-for-siw}},
	\\
	\alpha(u,\eta) &= \Omega(u,\eta) + \frac{a(u,\eta)}{2}.
\end{align*}
By Lemma~\ref{lem:existence-of-formula-for-siw}, every Weil generator in $W$ is of the form $\alpha(u,\eta)$ for some $u \in \sO_F^\times$ and $\eta \in T$. However, $\alpha(u,\eta)$ may not be a Weil generator (in fact, it may not even be integral) for every choice of $u \in \sO_F^\times$ and $\eta \in T$.

The first step in our proof is to characterize those $u \in \sO_F^\times$ and $\eta \in T$ for which $\alpha(u,\eta) \in W$. It turns out that this property only depends on $u$. Recall that every unit $u \in \sO_F^\times$ is of the form $\pm u_0^k$ where $k \in \ZZ$ and $u_0$ is a fundamental unit for $F$. We will show that $\alpha(u,\eta) \in W$ if and only if $k$ satisfies a certain congruence condition (see Corollary~\ref{cor:alpha-u-in-W-iff-k-congruence} below).

The next step in our proof is to compare $h(\alpha(u,\eta))$ to $h(u)$. We will show that if $\alpha(u,\eta) \in W$, then $h(\alpha(u,\eta)) \approx h(u)^2$ (see Lemma~\ref{lem:bound-alpha-by-u-squared} below). Combined with above, this reduces the problem of counting Weil generators of bounded height to counting units in $\sO_F$ of bounded height.

\begin{lemma}\label{lem:g2-alpha-wg-depends-only-on-u}
	$\alpha(u,\eta) \in W$ if and only if $\alpha(\pm u,\eta') \in W$ for all $\eta' \in T$.
\end{lemma}
\begin{proof}
	Because $F$ is a real quadratic field, $\#T = 2$. Let $\eta_1,\eta_2$ be the distinct elements of $T$. Then $\eta_2 = -\eta_1 + b$ for some $b \in \ZZ$. It follows from Definition~\ref{defn:weil-generator} that if $\beta \in K$, then $\beta \in W$ if and only if $\{\pm\beta,\pm\overline{\beta}\} \subseteq W$ (see also Example~\ref{ex:operations-on-weil-gens}).

	By Lemma~\ref{lem:existence-of-formula-for-siw}, every Weil generator $\beta \in W$ corresponds to a unique triple $u \in \sO_F^\times$, $\eta \in T$, and $a \in \ZZ$. It is not hard to see that $-\beta$ must correspond to the triple $u',\eta',a'$ where $u' = -u$ and $\eta' = -\eta + b$ for some integer $b$. Hence $-\alpha(u,\eta_1) = \alpha(-u,\eta_2)$. A similar argument shows that $\overline{\alpha(u,\eta_1)} = \alpha(-u,\eta_1)$. Hence $\alpha(u,\eta_1) \in W$ if and only if $\{\alpha(\pm u,\eta_1), \alpha(\pm u,\eta_2)\} \subseteq W$.
\end{proof}

\begin{lemma}\label{lem:alpha-in-W-if-u-in-congruence}
	There is a set of congruence classes $S \subseteq \sO_K/4\sO_K$ such that if $u \in \sO_K^\times$ and $\eta \in T$, then $\alpha(u,\eta) \in W$ if and only if $u\mod{4\sO_K} \in S$.
\end{lemma}
\begin{proof}
	By Lemma~\ref{lem:g2-alpha-wg-depends-only-on-u}, it suffices to show that if $u' \equiv u \mod{4\sO_K}$ then $\alpha(u,\eta) \in W$ if and only if $\alpha(u',\eta) \in W$. The set $S$ will be $\{u\mod{4\sO_K} \st \alpha(u,\eta) \in W \text{ for some $\eta \in T$} \}$.

	We will start by showing that $\alpha(u,\eta) \in W$ if and only if $4\alpha(u,\eta) \equiv 0 \mod{4\sO_K}$. Note that $\alpha(u,\eta)$ satisfies $\alpha(u,\eta)\overline{\alpha(u,\eta)} \in \QQ$, $\alpha(u,\eta) + \overline{\alpha(u,\eta)} = \eta + a$, and $\alpha(u,\eta) - \overline{\alpha(u,\eta)} = u\left(\gamma - \overline{\gamma}\right)$. It follows from Lemma~\ref{lem:equiv-defn-for-siw} that $\alpha(u,\eta) \in W$ if and only if $\alpha(u,\eta) \in \sO_K$. By our construction, $\alpha(u,\eta) \in (1/4)\sO_K$, so $\alpha(u,\eta) \in \sO_K$ if and only if $4\alpha(u,\eta) \equiv 0 \mod{4\sO_K}$.

	Next we will show that the equivalence class of $4\alpha(u,\eta) \mod{4\sO_K}$ depends only $u\mod{4\sO_K}$. It is sufficient to show that $2a(u,\eta)\mod{4}$ depends only on $u \mod{4\sO_K}$ as the dependence of $4\Omega(u,\eta)$ is clear. The construction of $a(u,\eta)$ given in Lemma~\ref{lem:uniqueness-of-formula-for-siw} depends only on the coefficients of $\Omega(u,\eta)\overline{\Omega(u,\eta)}$ with respect to the basis $\{1,\eta\}$. Suppose $u'$ is another unit with $u' = u + 4\beta$ for some $\beta \in \sO_K$. Then $\Omega(u',\eta) = \Omega(u,\eta) + 2(\gamma - \overline{\gamma})\beta$. So
\begin{align*}
	\Omega(u',\eta)\overline{\Omega(u',\eta)}
	&= \Omega(u,\eta)\overline{\Omega(u,\eta)} - 2(\gamma - \overline{\gamma})\beta(\Omega(u,\eta) + \overline{\Omega(u,\eta)}) - 4(\gamma - \overline{\gamma})^2\beta^2
	\\
	&= \Omega(u,\eta)\overline{\Omega(u,\eta)} - 2(\gamma - \overline{\gamma})\beta\eta - 4(\gamma - \overline{\gamma})^2\beta^2.
\end{align*}
	Looking at the coefficients with respect to the basis $\{1,\eta\}$ shows that $a(u',\eta) - a(u,\eta) \in 2\ZZ$. Therefore $4\alpha(u,\eta) \equiv 4\alpha(u',\eta) \mod{4\sO_K}$.
\end{proof}

Let $u_0$ be a fundamental unit for $F$. Then every unit $u \in \sO_F^\times$ is of the form $u = \pm u_0^k$.

\begin{corollary}\label{cor:alpha-u-in-W-iff-k-congruence}
	If $\pm u_0^k \in \sO_F^\times$ and $\eta \in T$, then $\alpha(\pm u_0^k, \pm\eta) \in W$ if and only if $k$ satisfies a certain congruence relation.
\end{corollary}
\begin{proof}
	This follows directly from Lemma~\ref{lem:g2-alpha-wg-depends-only-on-u} and Lemma~\ref{lem:alpha-in-W-if-u-in-congruence} using the fact that $u_0$ has finite multiplicative order in $(\sO_K/4\sO_K)^\times$.
\end{proof}

Next we will show that $h(\alpha) \approx h(u)^2$. In order to compare $h(\alpha)$ and $h(u)$, we will need the following lemmas. For $\eta \in T$, we can write any $\beta \in F$ in the form $\beta = b + c\eta$ for some unique $b,c \in \QQ$. The lemmas below will be used to compare $h(\beta)$ and $|c|$.

\begin{lemma}\label{lem:bound-hbeta-below-by-c}
	Let $\eta \in \overline{\QQ}\setminus\QQ$ and let $\beta = b + c\eta$ for some $b,c \in \QQ$. Then there is a positive constant $C_1$, depending only on $\eta$, such that if $c \neq 0$, then
	\[
		C_1|c| \leq h(\beta).
	\]
\end{lemma}
\begin{proof}
	Let $\sigma$ and $\tau$ be embeddings $\overline{\QQ} \to \CC$ such that $\sigma(\eta) \neq \tau(\eta)$. Let $C_1 = |\sigma(\eta) - \tau(\eta)|/2$, i.e. $C_1$ is half the distance from $\sigma(\eta)$ to $\tau(\eta)$. Since $-b/c$ cannot be closer than $C_1$ to both $\sigma(\eta)$ and $\tau(\eta)$, we have that
	\[
		C_1
		\leq \max\left\{\left|\frac{b}{c} + \sigma(\eta)\right|,\left|\frac{b}{c} + \tau(\eta)\right|\right\}
		\leq h\left(\frac{b}{c} + \eta\right).
	\]
	Up to replacing $b,c$ with $-b,-c$, we may assume that $c$ is positive. The claim then follows from multiplying this inequality by $c$ and using the fact that $h$ commutes with multiplication by a positive rational number.
\end{proof}

\begin{lemma}\label{lem:bound-hbeta-above-by-c}
	Let $F$ be a real quadratic field and let $\eta \in \sO_F$. There exists a positive constant $C_2$, depending only on $\eta$, such that if $\beta = b + c\eta$ with $b,c \in \ZZ$ and $c \neq 0$, then
	\[
		h(\beta) \leq C_2|c|\sqrt{\left|\norm_{F/\QQ}\beta\right|}.
	\]
\end{lemma}
\begin{proof}
	By the triangle inequality, $h(\beta) \leq |b| + |c|h(\eta)$. Therefore it is enough to bound $|b|$ from above by a constant times $|c| \sqrt{\left|\norm_{F/\QQ}(\beta)\right|}$. Let $t = \tr_{F/\QQ}(\eta)$, $n_{\eta} = \norm_{F/\QQ}(\eta)$, and $n_{\beta} = \norm_{F/\QQ}(\beta)$. Then we can write the norm of $b + c\eta = \beta$ as $b^2 + bct + c^2n_\eta = n_\beta$. Therefore
	\begin{align*}
		|b| &= \frac{\left|-ct \pm \sqrt{(ct)^2 - 4\left(c^2n_{\eta} - n_{\beta}\right)}\right|}{2}
		\\
		&\leq |c| \cdot \left(\frac{|t| + \sqrt{t^2 + 4\left(|n_{\eta}| + |n_{\beta}|/c^2\right)}}{2}\right)
		\\
		&\leq |c| \cdot \left(\frac{|t| + \sqrt{t^2 + 4\left(|n_{\eta}| + |n_{\beta}|\right)}}{2}\right)
		\\
		&\leq |c| \sqrt{|n_{\beta}|} \cdot \left(\frac{|t| + \sqrt{t^2 + 4\left(|n_{\eta}| + 1\right)}}{2}\right).
	\end{align*}
\end{proof}

Next we will combine the two previous lemmas to show that if $\alpha \in W$ and $u \in \sO_F^\times$ is the associated unit, then $h(\alpha)$ is approximately $h(u)^2$.

\begin{lemma}\label{lem:bound-alpha-by-u-squared}
	There exist positive constants $C_3,C_4,C_5$, depending only on $\gamma$ and $T$, such that if $u\in \sO_F^\times$, $\eta \in T$, and $h(u) \geq C_3$, then
	\[
		C_4h(u)^2 \leq h(\alpha(u,\eta)) \leq C_5h(u)^2.
	\]
\end{lemma}
\begin{proof}
	To prove the claim we will show that if $h(u)$ is sufficiently large, then $h(\Omega(u,\eta))$ is approximately $h(u)$ and $|a(u,\eta)|$ is approximately $h(u)^2$.

	First we will show that $h(\Omega(u,\eta))$ is approximately $h(u)$. By choosing $C_3$ large enough, we may assume that if $h(u) \geq C_3$, then
	\[
		\frac{\min_\sigma\left\{ |\sigma(\gamma - \overline{\gamma})| \right\}}{2}h(u)
		\leq
		h(u)\min_\sigma\left\{ |\sigma(\gamma - \overline{\gamma})| \right\} - h(\eta).
	\]
	Note that the right hand side is
	\[
		\leq
		h(u(\gamma - \overline{\gamma})) - h(\eta).
	\]
	By the triangle inequality, the last expression is
	\[
		\leq
		2h(\Omega(u,\eta))
		\leq
		h(u)h(\gamma - \overline{\gamma}) + h(\eta).
	\]
	We may also choose $C_3$ large enough so that $h(u) \geq h(\eta)$ for all $\eta \in T$. So the previous expression is
	\[
		\leq
		\left(h(\gamma - \overline{\gamma}) + 1\right)h(u).
	\]
	Thus we have shown that
	\[
		\frac{\min_\sigma\left\{ |\sigma(\gamma - \overline{\gamma})| \right\}}{2}h(u) \leq 2h(\Omega(u,\eta)) \leq \left(h(\gamma - \overline{\gamma}) + 1\right)h(u).
	\]

	Next we will show that $|a(u,\eta)|$ is approximately $h(u)^2$. Recall that $\alpha(u,\eta) = \Omega(u,\eta) + a(u,\eta)/2$. As seen in the proof of Lemma~\ref{lem:uniqueness-of-formula-for-siw}, $-a(u,\eta)/2$ is the coefficient of $\eta$ when $\Omega(u,\eta)\overline{\Omega(u,\eta)}$ is written with respect to the $\QQ$-basis $\{1,\eta\}$ of $F$. That is, $\Omega(u,\eta)\overline{\Omega(u,\eta)} = b - a(u,\eta)/2\eta$ for some $b \in \ZZ$.

	We would like to apply Lemmas~\ref{lem:bound-hbeta-below-by-c} and \ref{lem:bound-hbeta-above-by-c} to relate $|a(u,\eta)|$ to $h(\Omega(u,\eta)\overline{\Omega(u,\eta)})$. However, $\norm_{F/\QQ}(\Omega(u,\eta)\overline{\Omega(u,\eta)})$ may be large, so the bound in Lemma~\ref{lem:bound-hbeta-above-by-c} is not useful. To get around this issue, we will consider $\Omega(u,\eta) - \eta/2$ instead of $\Omega(u,\eta)$. Let $\beta(u,\eta) = (\Omega(u,\eta) - \eta/2)(\overline{\Omega(u,\eta)} - \eta/2)$. Note that
	\[
		\beta(u,\eta) = \frac{-u^2(\gamma - \overline{\gamma})^2}{4}
	\]
	and
	\[
		\beta(u,\eta)
		=
		b - \frac{a(u,\eta)}{2}\eta - \frac{\eta}{2}\left(\Omega(u,\eta) + \overline{\Omega(u,\eta)}\right) + \frac{\eta^2}{4}
		=
		b - \frac{a(u,\eta)}{2}\eta - \frac{\eta^2}{4}.
	\]
	The equation on the left shows that $\norm_{F/\QQ}\beta(u,\eta)$ depends only on $\gamma$, and that $h(\beta(u,\eta))$ can be bounded above and below by $h(u)^2$ times constants depending only on $\gamma$. The equation on the right shows that the coefficient of $\eta$ of $\beta(u,\eta)$ written with respect to the basis $\{1,\eta\}$ is $-a(u,\eta)/2$ plus a constant depending only on $\eta$. Therefore, assuming that $\beta(u,\eta) \notin \QQ$, we may apply Lemmas~\ref{lem:bound-hbeta-below-by-c} and \ref{lem:bound-hbeta-above-by-c} to relate $h(\beta(u,\eta))$ and $|a(u,\eta)|$. Up to replacing the constants in the lemmas by some factors that depend only on $\eta$ and $\gamma$, we have
	\[
		C_1|a(u,\eta)| \leq h(u)^2 \leq C_2|a(u,\eta)|,
	\]
	for some positive constants $C_1$ and $C_2$. Repeating the process and taking the minimum (resp. maximum) over all $\eta \in T$, we can replace $C_1$ and (resp. $C_2$) with constants that only depend on $\gamma$ and $T$ instead of $\gamma$ and $\eta$.

	The final step is to show that there are only finitely many $u \in \sO_F^\times$ and $\eta \in T$ such that such that $\beta(u,\eta) \in \QQ$. This is necessary because the hypotheses of Lemmas~\ref{lem:bound-hbeta-below-by-c} and \ref{lem:bound-hbeta-above-by-c} require that $c \neq 0$, where $c$ is the coefficient of $\eta$ of $\beta(u,\eta)$ written with respect to the basis $\{1,\eta\}$. Note that if $c = 0$, then $\beta(u,\eta) \in \QQ$ and so $h(\beta(u,\eta)) = \sqrt{\norm_{F/\QQ}(\beta(u,\eta))}$. The latter depends only on $\gamma$. However, we have already seen $h(\beta(u,\eta))$ is bounded below by a constant times $h(u)^2$. Hence there is a finite number of $u \in \sO_F^\times$ such that $\beta(u,\eta) \in \QQ$ for some $\eta \in T$.
\end{proof}

We are now ready to finish the proof of Proposition~\ref{prop:W-size-g=2} by counting the number of $u \in \sO_F^\times$ for which $\alpha(u,\eta) \in W$ for some $\eta \in T$.

\begin{proof}[Proof of Proposition~\ref{prop:W-size-g=2}]
	Let $u_0$ be a fundamental unit for $F$. By Lemma~\ref{lem:bound-alpha-by-u-squared}, there are positive constants $C_3,C_4,C_5$, which depend only on $\gamma$ and $T$, such that
	\[
		C_4h(u)^2 \leq h(\alpha(u,\eta)) \leq C_5h(u)^2
	\]
	for all $u \in \sO_F^\times$ and $\eta \in T$ with $h(u) \geq C_3$. Setting $u = \pm u_0^k$ for some $k \in \ZZ$, we can rewrite these inequalities as
	\[
		C_4h(u_0)^{2|k|} \leq h(\alpha(\pm u_0^k,\pm \eta)) \leq C_5h(u_0)^{2|k|}.
	\]

	By Corollary~\ref{cor:alpha-u-in-W-iff-k-congruence}, $\alpha(\pm u_0^k, \pm \eta) \in W$ if and only if $k$ satisfies some congruence condition. Let $\script{P}$ denote the set of such $k$. Let $S$ be the number of $\alpha(u,\eta) \in W$ such that $h(u) < C_3$. Using the inequalities above, we have
	\begin{align*}
		4 \cdot \#\left\{ k \in \script{P} \st C_5h(u_0)^{2|k|} \leq N \right\} - S
		&\leq
		\#\left\{ \alpha \in W \st h(\alpha) \leq N \right\}
		\\
		&\leq
		4 \cdot \#\left\{ k \in \script{P} \st C_4h(u_0)^{2|k|} \leq N \right\} + S.
	\end{align*}
	Let $C_6$ be the density of $\script{P}$ in $\ZZ$. Then the both sides of the above inequality are asymptotic to
	\[
		\frac{2C_6\log N}{\log h(u_0)} + O(1).
	\]
\end{proof}

\begin{example}\label{ex:asymptotic-value-number-W-N-zeta5}
	We will show how to compute the constant $\rho$ from Proposition~\ref{prop:W-size-g=2} for the field $K = \QQ(\zeta_5)$. Let $\gamma = \zeta_5$, $\eta_0 = \zeta_5 + \overline{\zeta}_5$, and $T = \{ \eta_0,-\eta_0 \}$.\footnote{It is not always true that $T$ can be chosen as fundamental units of $F$. For example, the fundamental units for $\QQ(\sqrt{6})$ are $\pm(5 + 2\sqrt{6})^{\pm 1}$ which do not generate the ring of integers $\ZZ[\sqrt{6}]$. Similarly, $\gamma + \overline{\gamma}$ is not always in $T$. For example, if $K = \QQ(\sqrt{5},\sqrt{-1})$ then $\sO_K = \ZZ[(1+\sqrt{5})/2,\sqrt{-1}]$ so we may choose $\gamma = \sqrt{-1}$ hence $\gamma + \overline{\gamma} = 0$.} A fundamental unit for $F$ is $u_0 = \zeta_5 + \overline{\zeta}_5$.

	Note that it is not always the case that $a(u,\eta) \in \ZZ$. For example, if $u = \eta_0^2$, then $a(u,\eta_0) = 5/2$. This shows that $\alpha(u,\eta)$ is not always a Weil generator.

	Next we will determine explicitly the condition on $k$ such that $\alpha(u_0^k,\eta_0) \in W$. One can show that $u_0$ has order $6$ in $(\sO_F/2\sO_F)^\times$. Table~\ref{tbl:residue-classes-phi-u-eta-slim} gives the values of $4\alpha(u_0^k,\eta_0)\mod{4\sO_K}$ for $k = 0,1,\dots,5$. It shows that $\alpha(\pm u_0^k,\pm\eta_0) \in W$ if and only if $k \not\equiv 2\mod{3}$.

	\begin{table}[h]
		\centering
		%
		%
		\begin{tabular}{c|c}
			$k$ & $4\alpha(u_0^k,\eta_0) \mod 4\sO_K$ \\\hline
			$0$ & $0$ \\
			$1$ & $0$ \\
			$2$ & $-2\zeta_5^3 - 2\zeta_5^2 + 1$ \\
			$3$ & $0$ \\
			$4$ & $0$ \\
			$5$ & $-2\zeta_5^3 - 2\zeta_5^2 + 1$
		\end{tabular}
		\caption{Values of $4\alpha(u_0^k,\eta_0)$ modulo $4\sO_K$.}\label{tbl:residue-classes-phi-u-eta-slim}
	\end{table}

	We have shown that each $k \in \ZZ$ with $k \not\equiv 2\mod{3}$ corresponds to $4$ Weil generators given by $\alpha(\pm u_0^k, \pm \eta_0)$. By Lemma~\ref{lem:bound-hbeta-below-by-c}, $h(\alpha(u,\eta)) \approx h(u)^2$. Since and each $k \not\equiv 2\mod{3}$ corresponds to $4$ Weil generators given by $\alpha(\pm u_0^k,\pm \eta_0)$, we have
	\begin{align*}
		\#\left\{ \alpha \in W \st h(\alpha) \leq N \right\}
		&=
		4\cdot\#\left\{ k \in \ZZ \st k\not\equiv 2\mod{3} \text{ and } h(u_0^k) \leq \sqrt{N} \right\} + O(1)
		\\
		&=
		4\cdot\#\left\{ k \in \ZZ \st k\not\equiv 2\mod{3} \text{ and } |k| \leq \frac{\log N}{2\log h(u_0)} \right\} + O(1)
		\\
		&= \frac{8\log N}{3\log h\left(u_0\right)} + O(1).
		\\
		&= \frac{8\log N}{3\log \left(\frac{1 + \sqrt{5}}{2}\right)} + O(1).
	\end{align*}

	Figure~\ref{fig:weil-gens-less-than-N-and-8/3-log-N} below shows the accuracy of this estimate for the number of Weil generators of bounded height.

	\begin{figure}[h]
		\centering
		%
		%
		\includegraphics[width=\linewidth]{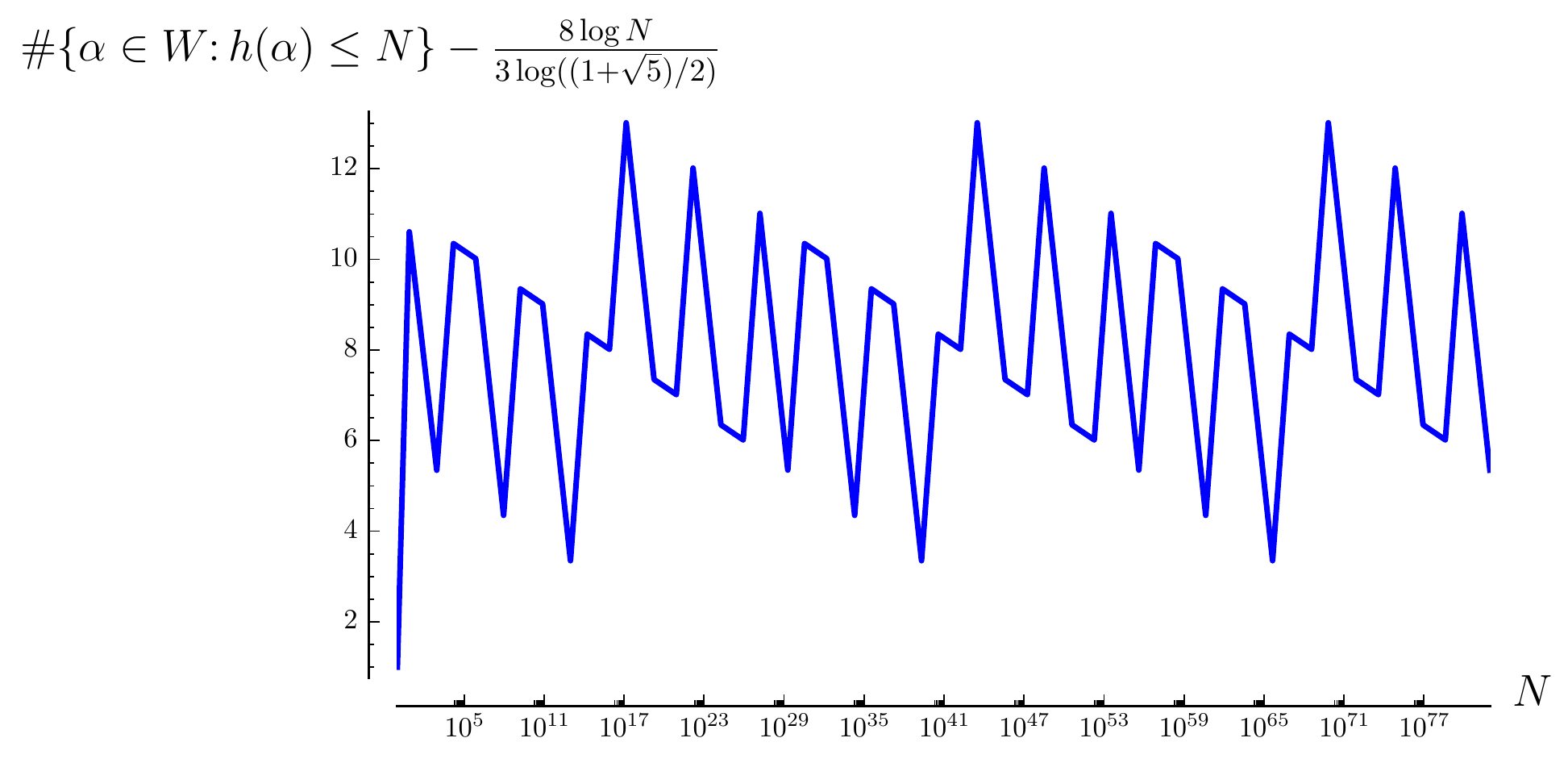}
		\caption{A comparison of the number of Weil generators of bounded height as found by Algorithm~\ref{alg:find-siw-g>=2} and the asymptotic value given in Example~\ref{ex:asymptotic-value-number-W-N-zeta5}.}
		\label{fig:weil-gens-less-than-N-and-8/3-log-N}
	\end{figure}
\end{example}

\subsection{The Case $g \geq 3$}\label{sec:g>=3}

The goal of this section is to prove Theorem~\ref{thm:W-size} in the case $g \geq 3$.

\begin{proposition}\label{prop:W-size-g>=3}
	If $K$ is a CM field of degree $2g$ with $g \geq 3$, then $W$ is finite. Moreover, if $g = 3$ then there is a computable upper bound for $\#W$.
\end{proposition}

Our strategy for Proposition~\ref{prop:W-size-g>=3} is as follows. Recall that every Weil generator $\alpha$ can be written in the form $(u(\gamma - \overline{\gamma}) + \eta + a)/2$ for a unique $u \in \sO_F^\times$, $\eta \in T$, and $a \in \ZZ$. Recall that by definition, $\alpha\overline{\alpha} \in \ZZ$. This condition places a significant restriction on the possible values of $u$, $\eta$, and $a$. By Lemma~\ref{lem:uniqueness-of-formula-for-siw}, $a$ is uniquely determined by $u$ and $\eta$. Therefore it suffices to show that the possible set of pairs $(u,\eta)$ arising in this way is finite. In fact, we will parameterize pairs $(u,\eta)$ by integral points on a finite union of absolutely irreducible plane curves of degree $g$. These curves will have $g$ distinct points at infinity. So by Siegel's theorem, the number of integral points is finite. When $g = 3$, the curves have genus $1$ or $0$ depending on the singularities. In the genus $0$ case, finding integral points reduces to finding solutions to an $S$-unit equation, which can be effectively determined \cite[Thm.~D.8.4]{hindry2000diophantine}. In the genus $1$ case, we may use the effective (but impractical) bounds from Baker and Coates \cite{baker1970integer}. For more details on the effective bounds, see Section~\ref{sec:effectiveness}.

We start by proving a lemma which will be used to show that the curves arising in the proof of Proposition~\ref{prop:W-size-g>=3} are geometrically irreducible.

\begin{lemma}\label{lem:lines plus constant irreducible}
	Let $f_1,\dots,f_g \in \overline{\QQ}[x,y,z]$ be homogeneous linear polynomials such that the lines in projective space defined by the vanishing of the $f_i$ intersect the line at infinity (given by $z = 0$) at distinct points. If $t \in \overline{\QQ}$ is nonzero, then $tz^g + \prod f_i$ is irreducible.
\end{lemma}
\begin{proof}
	Let $F_t = tz^g + \prod f_i$, and let $h_1,\dots,h_k$ be the irreducible factors of $F_t$, i.e.
	\[
		F_t(x,y,z) = tz^g + \prod_{i = 1}^{g}f_i(x,y,z) = \prod_{j=1}^{k} h_j(x,y,z).
	\]
	By hypothesis, the line defined by $f_i(x,y,z) = 0$ does not coincide with the line at infinity. So by a change of coordinates fixing the variable $z$, we may assume that $f_1 = x$. Then $F_t(0,y,z) = tz^g = \prod h_j(0,y,z)$. Because $t \neq 0$, it follows that for all $j$, $h_j(0,y,z)$ is of the form $a_jz^{n_j}$ for some nonzero $a_j$ and positive integer $n_j$. This shows that the point $(0:1:0)$ lies in every irreducible component of the projective variety defined by $F_t$.

	Notice that the projective variety defined by $F_0$ is a union of lines, all of whose pairwise intersections occur in the affine plane by the hypothesis on the $f_i$. This means that $F_0$ has no singularities on the line at infinity. The same property also holds for $F_t$ because
	\[
		\frac{dF_t}{dx} = \frac{dF_0}{dx},
		\quad
		\frac{dF_t}{dy} = \frac{dF_0}{dy},
		\quad
		\frac{dF_t}{dz}(x,y,0) = \frac{dF_0}{dz}(x,y,0),
		\quad
		\text{and }
		F_t(x,y,0) = F_0(x,y,0).
	\]
	Therefore the point $P = (0:1:0)$ must be a smooth point of the variety defined by $F_t$, and hence lies in a unique irreducible component. But by the above, $P$ lies in every irreducible component, hence $F_t$ is irreducible.
\end{proof}

We are now ready to prove Proposition~\ref{prop:W-size-g>=3}.

\begin{proof}[Proof of Proposition~\ref{prop:W-size-g>=3}]
	Let $\alpha \in W$. Recall that $\alpha$ can be written as $(u(\gamma - \overline{\gamma}) + \eta + a)/2$ for a unique $u \in \sO_F^\times$, $\eta \in T$, and $a \in \ZZ$. Let $\Omega = \alpha - a/2$. By definition,
	\[
		\alpha\overline{\alpha} = \Omega\overline{\Omega} + \frac{\eta a}{2} + \frac{a^2}{4} \in \ZZ.
	\]
	Let $\delta = (\gamma - \overline{\gamma})(\overline{\gamma} - \gamma)$. Then we have shown that
	\begin{equation}\label{eq:omega-omegabar-eq-in-terms-of-A-and-B}
		4\Omega\overline{\Omega} = u^2\delta + \eta^2 = A + B\eta
	\end{equation}
	for some $A,B \in \ZZ$. It is important that $4\Omega\overline{\Omega}$ lies in the $\ZZ$-span of $\{1,\eta\}$. This already is a significant restriction on the possible values of $u$ and $\eta$ since $\sO_F$ has rank $g \geq 3$ by hypothesis.

	By \cite[Cor.~2.10, Pg.~202]{neukrich1999algebraic}, $\norm_{F/\QQ}(u^2\delta) = \norm_{F/\QQ}(\delta) = \pm\disc_{K}/\disc_{F}^2$. So by rearranging equation~(\ref{eq:omega-omegabar-eq-in-terms-of-A-and-B}) and taking norms, we have that
	\[
		\norm_{F/\QQ}\left(A + B\eta - \eta^2\right) = \frac{\left|\disc_{K}\right|}{\disc_{F}^2}.
	\]
	In particular, $(x,y)=(A,B)$ is an integral point on the affine curve $C_\eta$ given by the vanishing of the polynomial
	\begin{equation}\label{eq:C_eta-curve}
		-\frac{\left|\disc_{K}\right|}{\disc_{F}^2} + \prod_{\sigma:F \to \CC}\left(x + y\sigma(\eta) - \sigma(\eta)^2\right).
	\end{equation}
	By construction, the projective closure of $C_\eta$ has $g$ distinct points at infinity given by $(-\sigma(\eta) : 1 : 0)$. Therefore we may apply Lemma~\ref{lem:lines plus constant irreducible}, which implies that $C_\eta$ is geometrically irreducible. It follows from Siegel's theorem \cite{siegel1929some} (see also \cite[Rem.~D.9.2.2]{hindry2000diophantine}) that $C_\eta$ has finitely many integral points.

	To finish the proof, it suffices to show that the map described above sending $\alpha \in W$ to the integral point $(A,B)$ is finite-to-one. Recall that by Theorem~\ref{thm:finite-monogenerators}, the set $T$ is finite, so we may fix some $\eta \in T$. By equation~(\ref{eq:omega-omegabar-eq-in-terms-of-A-and-B}), $u$ is determined up to sign by the point $(A,B)$, $\eta$, and $\delta$. Finally, by Lemma~\ref{lem:uniqueness-of-formula-for-siw}, $a$ is determined by the pair $(u,\eta)$. Hence for each point $(A,B)$ there is a finite number of possible triples $(u,\eta,a)$ corresponding to Weil generators.

	When $g = 3$, we use an effective version of Siegel's theorem. Note that in this case, the curve $C_\eta$ has degree $3$, so it has genus $1$ or $0$ depending on its singularities. If $C_\eta$ has genus $0$, then the integral points can be computed by solving an $S$-unit equation \cite[Thm.~D.8.4]{hindry2000diophantine}. If $C_\eta$ has genus $1$, then the main theorem of \cite{baker1966linear} gives a computable (but impractical) bound on the number of integral points. For more details, see Section~\ref{sec:effectiveness}.
\end{proof}

\section{Effectiveness When $g = 3$}\label{sec:effectiveness}

The goal of this section is to show how to give a concrete bound for the number of Weil generators in a sextic CM field. We will start by summarizing the relevant results from Section~\ref{sec:counting-weil-gens}.

Let $K$ be a CM field with maximal totally real subfield $F$ and let $W$ be the set of all Weil generators in $K$. Recall from Lemma~\ref{lem:existence-of-formula-for-siw} that every $\alpha \in W$ corresponds to a unique triple $(u,\eta,a)$ with $u \in \sO_F^\times$, $\eta \in T$, and $a \in \ZZ$. By Lemma~\ref{lem:uniqueness-of-formula-for-siw}, $a$ is uniquely determined by $\eta$ and $u$. Furthermore, the proof of Proposition~\ref{prop:W-size-g>=3} showed that all possible values of $u$ for a fixed $\eta$ are determined up to sign by the integral points of the curve $C_\eta$ defined by equation~(\ref{eq:C_eta-curve}). Therefore,
\[
	\#W \leq 2\sum_{\eta \in T} \#C_{\eta}(\ZZ).
\]
The following lemma shows that if $\deg K = 6$, then it is sufficient to consider a single $\eta \in T$.

\begin{lemma}\label{lem:eta-curves-are-isomorphic-g=3}
	Suppose that $\deg K = 6$, and let $\eta_1,\eta_2 \in T$. Then $C_{\eta_1}$ and $C_{\eta_2}$ are isomorphic via an integral linear change of variables. In particular, there is a bijection between $C_{\eta_1}(\ZZ)$ and $C_{\eta_2}(\ZZ)$. In particular, $\# W \leq 2 \cdot \#T \cdot \#C_{\eta_1}(\ZZ)$.
\end{lemma}
\begin{proof}
	Recall that the curve $C_{\eta_i}$ is defined by the polynomial $f_{\eta_i}(x,y)$ in equation~(\ref{eq:C_eta-curve}). For any embedding $\sigma: F \to \RR$, let $L_i^\sigma(x,y) = x + y\sigma(\eta_i) - \sigma(\eta_i)^2$. Then
	\[
		f_{\eta_i}(x,y) = \prod_{\sigma:F \to \RR}L_i^\sigma(x,y) - \frac{|\disc_{K}|}{\disc_{F}^2}.
	\]

	To prove the claim, we will construct an invertible integral change of coordinates $\tilde{v}$ such that $f_{\eta_1} \circ \tilde{v} = f_{\eta_2}$. Our construction of $\tilde{v}$ relies on finding an algebraic integer $v \in \sO_F$ with the property that
	\begin{equation}\label{eq:reqs-for-u-in-lemma}
		v = A_1 + B_1\eta_1, \quad
		v\eta_2 = A_2 + B_2\eta_1, \quad
		v\eta_2^2 = A_3 + B_3\eta_1 + \eta_1^2,
	\end{equation}
	for some integers $A_i,B_i \in \ZZ$. First we will show how to construct $\tilde{v}$ given such a $v$. A construction of $v$ is given at the end of the proof.

	We define $\tilde{v}:\RR^2 \to \RR^2$ by
	\[
		\tilde{v}(x,y) = (A_1x + A_2y - A_3, B_1x + B_2y - B_3).
	\]
	To show that $\tilde{v}$ has the desired property, we will first show that $\tilde{v}$ transforms $L_1^\sigma$ to a scaled $L_2^\sigma$. Note that
	\begin{align*}
		(L_1^\sigma \circ \tilde{v})(x,y)
		&=
		(A_1x + A_2y - A_3) + (B_1x + B_2y - B_3)\sigma(\eta_1) - \sigma(\eta_1^2) \\
		&=
		\sigma(A_1 + B_1\eta_1)x + \sigma(A_2 + B_2\eta_1)y - \sigma(A_3 + B_3\eta_1 + \eta_1^2)
		\\
		&=
		\sigma(v)\left(x + \sigma(\eta_2)y - \sigma(\eta_2^2)\right)
		\\
		&=
		\sigma(v)L_2^\sigma(x,y).
	\end{align*}
	Therefore
	\[
		f_{\eta_1} \circ \tilde{v}
		=
		\prod_{\sigma} \left(L^{\sigma}_1\circ \tilde{v}\right) - \frac{|\disc_K|}{\disc_F^2}
		=
		\norm_{F/\QQ}(v)\prod_{\sigma} L^\sigma_2 - \frac{|\disc_K|}{\disc_F^2}.
	\]
	If $v \in \sO_F^\times$, then the right hand side of this equation is $f_{\eta_2}$ as required.

	Next we will show that $v \in \sO_F^\times$ by constructing an inverse to $\tilde{v}$. Because $L_1^\sigma \circ \tilde{v} = \sigma(v)L_2^\sigma$ for all $\sigma$, it follows that $\tilde{v}$ maps the intersection of the lines defined by $L_2^{\sigma_i}(x,y) = 0$ and $L_2^{\sigma_j}(x,y) = 0$ to the intersection of the lines defined by $L_1^{\sigma_i}(x,y) = 0$ and $L_1^{\sigma_j}(x,y) = 0$, for all pairs $(\sigma_i,\sigma_j)$. By swapping $\eta_1$ and $\eta_2$ in our construction for $v$, which is given below, we can find an element $v' \in \sO_F$ such that the linear map $\tilde{v}'$ acts as an inverse to $\tilde{v}$ on these intersections. The lines defined by the polynomials $\{L_1^\sigma \st \sigma:F \to \RR \}$ are in general position.
	%
	So it follows that $\tilde{v}\circ\tilde{v}'$ fixes three distinct points. Therefore, $\tilde{v}$ and $\tilde{v}'$ are inverses, and
	\begin{align*}
		L_1^\sigma
		= L_1^\sigma \circ \tilde{v} \circ \tilde{v}'
		= \sigma(v)\left(L_2^\sigma \circ \tilde{v}'\right)
		= \sigma(vv')L_1^\sigma.
	\end{align*}
	This shows that $vv' = 1$, so $v \in \sO_F^\times$.

	It remains to construct the element $v$ satisfying equation~(\ref{eq:reqs-for-u-in-lemma}). Our (rather technical) construction of $v$ is as follows. Consider the two $\ZZ$-bases for $\sO_F$ given by $\script{B}_i = \left\{1,\eta_i,\eta_i^2\right\}$ for $i=1,2$. Let $P$ be the change of basis matrix from $\script{B}_2$ to $\script{B}_1$. That is, the columns of $P$ are the elements of $\script{B}_2$ written as vectors with respect to the basis $\script{B}_1$. Let $P^{-1}_{i,j}$ denote the $i,j$-entry of the matrix $P^{-1}$, and for $\beta \in \sO_F$, let $[\beta]_i$ denote the $i$th component of the vector given by writing $\beta$ with respect to the basis $\script{B}_1$. Now define
	\[
		v = P^{-1}_{3,3} - P^{-1}_{3,2}[\eta_1^3]_{3} + P^{-1}_{3,2}\eta_1.
	\]
	A straightforward but tedious calculation proves that $v$ satisfies the properties given in equation~(\ref{eq:reqs-for-u-in-lemma}).
\end{proof}

\begin{remark}
	Lemma~\ref{lem:eta-curves-are-isomorphic-g=3} fails when $g > 3$. For example, using {\tt Magma}, one can show that if $K = \QQ(\zeta_{15})$, $\eta_1 = \zeta_{15}^7 - \zeta_{15}^6 - \zeta_{15}^5 + 2\zeta_{15}^4 - \zeta_{15}^2 - 2$, and $\eta_2 = -\zeta_{15}^7 + \zeta_{15}^5 - \zeta_{15}^4 + \zeta_{15}^2 - \zeta_{15} - 3$; then $C_{\eta_1}$ and $C_{\eta_2}$ are not isomorphic.
	%
	%
	%
	%
	%
	%
	%
	%
\end{remark}

By Lemma~\ref{lem:eta-curves-are-isomorphic-g=3}, in order to find all Weil generators in a sextic CM field, it suffices to find all integral points on $C_\eta$ for a single $\eta \in T$. The others can be computed by the change of coordinates from the lemma. Because $g = 3$, $C_\eta$ is a plane curve of degree $3$, and so has genus $0$ or $1$. If $C_\eta$ is singular, then it has genus $0$ and the integral points can be enumerated by solving a certain $S$-unit equation \cite[Thm.~D.8.4]{hindry2000diophantine}. If $C_\eta$ is smooth, then it has genus $1$. In this case, one can attempt to find all integral points using the methods of \cite{stroeker2003computing}.

\subsection{An Example of Genus $0$}\label{sec:example-finding-weil-gens-in-field-with-genus0-Ceta}

In this section, we will find all Weil generators in $\QQ(\zeta_9)$. This field was chosen because it was the only sextic CM field with class number $1$ such that the resulting curves $C_\eta$ had genus $0$. The class number requirement is not used in this section, but it is relevant for finding super-isolated abelian varieties as described in Section~\ref{sec:super-isolated-variaties} below.

\begin{proposition}\label{prop:weil-gens-in-zeta9}
	There are $36$ Weil generators in $\QQ(\zeta_9)$. They are:
	{\tiny\begin{align*}
		&
		-3\zeta_9^4 - 2\zeta_9,\
		-\zeta_9^5 + 2\zeta_9^2,\
		-2\zeta_9^5 - \zeta_9^4 + 2\zeta_9^3 - 2\zeta_9 + 4,\
		\zeta_9^5 - 2\zeta_9^4 - 2\zeta_9^3 + 2\zeta_9^2 + 2,\
		-2\zeta_9^5 - 3\zeta_9^2,\
		\zeta_9^4 + 3\zeta_9,\
		\\
		&
		-2\zeta_9^5 + 2\zeta_9^4 - 2\zeta_9^3 - \zeta_9^2 + 2\zeta_9 + 2,\
		-\zeta_9^4 + 2\zeta_9^3 - 2\zeta_9^2 + \zeta_9 + 4,\
		-\zeta_9,\
		\zeta_9^5 + \zeta_9^2,\
		-\zeta_9^2,\
		\zeta_9^4 + \zeta_9,\
		\\
		&
		-\zeta_9^4 - \zeta_9,\
		\zeta_9^2,\
		-\zeta_9^5 - \zeta_9^2,\
		\zeta_9,\
		-\zeta_9^5 + 2\zeta_9^4 + 2\zeta_9^3 - 2\zeta_9^2 - 2,\
		2\zeta_9^5 + \zeta_9^4 - 2\zeta_9^3 + 2\zeta_9 - 4,\
		\\
		&
		\zeta_9^5 - 2\zeta_9^2,\
		3\zeta_9^4 + 2\zeta_9,\
		\zeta_9^4,\
		\zeta_9^5,\
		-\zeta_9^4 - 3\zeta_9,\
		2\zeta_9^5 + 3\zeta_9^2,\
		\zeta_9^4 - 2\zeta_9^3 + 2\zeta_9^2 - \zeta_9 - 4,\
		\\
		&
		2\zeta_9^5 - 2\zeta_9^4 + 2\zeta_9^3 + \zeta_9^2 - 2\zeta_9 - 2,\
		\zeta_9^5 - 2\zeta_9^3 - \zeta_9^2 - 2\zeta_9 + 2,\
		2\zeta_9^5 + 2\zeta_9^4 + 2\zeta_9^3 + 2\zeta_9^2 + \zeta_9 + 4,\
		2\zeta_9^4 - \zeta_9,\
		\\
		&
		3\zeta_9^5 + \zeta_9^2,\
		-2\zeta_9^5 - 2\zeta_9^4 - 2\zeta_9^3 - 2\zeta_9^2 - \zeta_9 - 4,\
		-\zeta_9^5 + 2\zeta_9^3 + \zeta_9^2 + 2\zeta_9 - 2,\
		-3\zeta_9^5 - \zeta_9^2,\
		-2\zeta_9^4 + \zeta_9,\
		-\zeta_9^5,\
		-\zeta_9^4.
	\end{align*}}
\end{proposition}

To prove Proposition~\ref{prop:weil-gens-in-zeta9}, we will find all integral points on $C_\eta$ for some $\eta \in T$. Then we will apply Lemma~\ref{lem:eta-curves-are-isomorphic-g=3} to find the integral points of $C_{\eta'}$ for all other $\eta' \in T$.

Let $\eta = \zeta_9 + \zeta_9^{-1}$. Then $C_\eta$ is the curve defined by the polynomial
\[
	f_\eta(x,y) = x^3 - 3xy^2 - y^3 - 6x^2 - 3xy + 9x + 3y - 4.
\]

\begin{lemma}\label{lem:points-on-C-eta-in-zeta9}
	Let $C_\eta$ be the plane curve defined above. Then $C_\eta(\ZZ)$ consists of the following ten points: $(22, -63)$, $(1, 0)$, $(1, -3)$, $(-2, 6)$, $(7, -3)$, $(4, 0)$, $(-2, -3)$, $(43, 21)$, $(-2, 3)$, $(-62, 42)$.
\end{lemma}
\begin{proof}
	Our proof follows the argument of \cite[Thm.~D.8.4]{hindry2000diophantine}. The main difference is that we will reduce the problem of finding integral points to solving a unit equation in $\sO_K$ instead of a more general $S$-unit equation.

	Let $\overline{C}_\eta$ denote the projective closure of $C_\eta$. Our first goal is to find a parameterization $\varphi: \PP^1 \to \overline{C}_\eta$. Note that $\overline{C}_\eta$ has a unique singular point $Q = (1:0:1)$. Every line $L$ through $Q$ intersects $\overline{C}_\eta$ at a unique point $P$. There is a bijection between the set of lines through $Q$ and $\PP^1$. This map is given by
	\begin{align*}
		\varphi(u,v) = \left(\varphi_x,\varphi_y,\varphi_z\right) = \left(
		u^3 - 3uv^2 + 4v^3,
		-3\left(u^3 - u^2v + uv^2\right),
		u^3 - 3u^2v + v^3
		\right).
	\end{align*}
	The inverse is
	\[
		\psi(x,y,z) = (\psi_u,\psi_v) = (y,z-x).
	\]

	Over $K$, $\varphi_y$ and $\varphi_z$ factor as
	\[
		\varphi_y = -3\prod_{i=1}^3 u - \alpha_iv,\quad
		\varphi_z = \prod_{i=1}^3 u - \beta_iv
	\]
	where $\alpha_i,\beta_i \in \sO_K$ and are all distinct.

	Suppose that $(u:v) \in \PP^1$ is such that $\varphi(u,v) \in C_\eta(\ZZ)$, i.e. $\varphi_x(u,v)/\varphi_z(u,v)$ and $\varphi_y(u,v)/\varphi_z(u,v)$ lie in $\ZZ$. We may assume that $u$ and $v$ are integral and coprime. Note that $3$ is totally ramified in $K$, so there is a unique prime of $K$ lying over $3$ which is generated by some $\nu \in \sO_K$. Let $S$ be the set containing only this prime. Next we will show that $u - \beta_jv$ is an $S$-unit. Note that
	\[
		\gcd\left(u - \alpha_iv,u - \beta_jv\right)
		\mid (\alpha_i - \beta_j)\gcd\left(u,v\right)
		= \alpha_i - \beta_j.
	\]
	A short computation shows that $\norm_{K/\QQ}(\alpha_i - \beta_j) \in \{1,9\}$, in particular, they are $S$-units. By hypothesis,
	\[
		\frac{-3\prod u - \alpha_iv}{\prod u - \beta_jv} = \frac{\varphi_y(u,v)}{\varphi_z(u,v)} \in \ZZ.
	\]
	By the above, the denominator is relatively prime to the numerator at all primes outside of $S$. Because the quotient is integral, this implies that the $u - \beta_jv$ are $S$-units.

	Next we claim that $\ord_{\nu}(u - \beta_j v) \leq 12$ for every $j$. Note that since $\varphi_y(u,v)/\varphi_z(u,v) \in \ZZ$, it follows that
	\[
		\sum_j \ord_{\nu}\left(u - \beta_j v\right)
		\leq
		6 + \sum_i \ord_{\nu}\left(u - \alpha_i v\right).
	\]
	Moreover, above we saw that for any fixed $i,j$, we have that $\gcd(u - \alpha_i v,u - \beta_j v) \mid \alpha_i - \beta_j$. A straightforward computation
	shows that $\max_{i,j}\{ \ord_{\nu}\left(\alpha_i - \beta_j\right) \} = 2$, so
	\[
		\min\left\{
			\ord_{\nu}\left(u - \alpha_i v\right),
			\ord_{\nu}\left(u - \beta_j v\right)
		\right\}
		\leq 2.
	\]
	So if $\ord_{\nu}(u - \beta_j v) > 12$, then by the first inequality, there is some $i$ such that $\ord_{\nu}(u - \alpha_i v) > 2$, but this contradicts the second inequality.

	Now we will show how finding integral points on $C_\eta$ reduces to solving a unit equation. Let $A = (\beta_2 - \beta_3)/(\beta_2 - \beta_1)$ and $B = (\beta_3 - \beta_1)/(\beta_2 - \beta_1)$. One can show
	that $A,B$ are units in $\sO_K^\times$. Then
	\[
		A \frac{u - \beta_1v}{u - \beta_3v} + B \frac{u - \beta_2v}{u - \beta_3v} = 1.
	\]
	This is sometimes called Siegel's identity. By the above, each summand in this equation is an $S$-unit whose valuation at $\nu$ is bounded between $-12$ and $12$.

	Therefore, we are looking for solutions $X,Y \in \sO_{K,S}^\times$ to the $S$-unit equation
	\[
		X + Y = 1.
	\]
	Given such a solution, we can solve for $u,v$ using the equations
	\[
		X = A \frac{u - \beta_1v}{u - \beta_3v}, \quad
		Y = B \frac{u - \beta_2v}{u - \beta_3v}.
	\]
	At this moment, we do not know of a widely available and refereed implementation of an $S$-unit equation solver over number fields for arbitrary sets $S$ of primes\footnote{One is currently being written for {\tt Sage}, see \url{https://trac.sagemath.org/ticket/22148}.}. However, because of the bounds on $\ord_{\nu}(u - \beta_j v)$, it suffices to find solutions $X,Y \in \sO_K^\times$ to the unit equation
	\[
		\nu^iX + \nu^jY = 1
	\]
	for all pairs $i,j$ with $-12 \leq i,j \leq 12$. Using {\tt Magma} \cite{magma},
	we found the $10$ integral points on $C_\eta$ listed in the statement.
\end{proof}

\begin{proof}[Outline of the Proof of Proposition~\ref{prop:weil-gens-in-zeta9}]
	The proof is computational, so we only outline the steps. For each $\eta' \in T$, we computed the transformation between $C_{\eta}$ and $C_{\eta'}$ given in the proof of Lemma~\ref{lem:eta-curves-are-isomorphic-g=3}. By applying these transformations to the points given in Lemma~\ref{lem:points-on-C-eta-in-zeta9}, we obtained $C_{\eta'}(\ZZ)$ for all $\eta'$. Finally, we used the method outlined in the proof of Proposition~\ref{prop:W-size-g>=3} to find all Weil generators associated to the integral points of $C_{\eta'}$.
\end{proof}

\subsection{An Example of Genus $1$}

Suppose that $C_\eta$ has genus $1$. By construction, the projective closure $\overline{C}_\eta$ of $C_\eta$ contains the $F$-rational point $(-\eta:1:0)$. Therefore $\overline{C}_\eta$ is isomorphic (over $F$) to an elliptic curve. If the rank of this curve is $0$, then we can provably find all Weil generators in $K$ by finding the torsion points on the elliptic curve.

\begin{example}\label{ex:no-q-points-on-C_eta}


	Let $K = \QQ(\beta)$ where $\beta$ is a root of $x^6 - x^5 + 3x^4 + 5x^2 - 2x + 1$. Let $\eta$ be a root of $x^3 + x^2 - 2x - 1$ in the maximal totally real subfield $F$ of $K$. Then $C_\eta$ is given by the polynomial
	\[
		f_\eta = x^3 + x^2y - 2xy^2 - y^3 - 5x^2 - xy + y^2 + 6x + 2y - 28.
	\]
	Over $F$, this curve is isomorphic to the elliptic curve $E$ given by the Weierstrass equation
	\[
		y^2 + xy + y = x^3 + 611x + 6416.
	\]
	We used {\tt Sage} to compute that $E/F$ has rank $0$. The torsion group $E(F)$ consists of the points $\{(0 : 1 : 0), (4 : 92 : 1), (4 : -97 : 1)\}$.
	In this case, $F$ is Galois and the images of $E(F)$ in $\overline{C}_\eta$ are the points at infinity, i.e. $(-\sigma(\eta) : 1 : 0)$ for each $\sigma \in \gal(F/\QQ)$. Therefore $C_\eta(\ZZ) = \emptyset$. By Lemma~\ref{lem:eta-curves-are-isomorphic-g=3}, the same holds for any $\eta' \in F$ such that $\sO_F = \ZZ[\eta']$, hence $K$ has no Weil generators.
\end{example}

\subsection{General Bounds for the Genus $1$ Case}

While there are general methods for finding integral points on genus $1$ curves \cite{stroeker2003computing}, they usually require starting with a rational point or a basis for the Mordell-Weil group over $F$. This may be too difficult to compute. However, we can give an upper bound on the height of any Weil generator for $K$ using a result of Baker and Coates \cite{baker1970integer}. This in turn can be used to bound $\#W$. In this section, we will compute this bound explicitly.

Recall from the proof of Proposition~\ref{prop:W-size-g>=3} that any Weil generator $\alpha$ corresponds to a point on $C_\eta$ as follows. We write $\alpha = \Omega + a/2$ where $\Omega = (u(\gamma - \overline{\gamma}) + \eta)/2$ and $\eta \in T$, $u \in \sO_F^\times$, and $a \in \ZZ$. Then $4\Omega\overline{\Omega} = (\eta^2 + u^2\delta) = A + B\eta$ and $P = (A,B)$ is an integral point on $C_\eta$.

Because $\alpha\overline{\alpha} = \Omega\overline{\Omega} + a\eta/2 + a^2/4 \in \ZZ$, we know that $a = -B/2$. So
\[
	h(\alpha)^2
	= \alpha\overline{\alpha}
	= \frac{A}{4} + \frac{B^2}{16}
	\leq h(P)^2,
\]
where $h(P) = \max(|A|,|B|)$.

Let $H_\eta$ denote the maximum absolute value of the coefficients of the defining polynomial of $C_\eta$ as given in equation~(\ref{eq:C_eta-curve}). By a theorem of Baker and Coates \cite{baker1970integer}, if $Q \in C_\eta(\ZZ)$, then
\[
	h(Q) \leq \exp\exp\exp\left(2H_\eta\right)^{10^{3^{10}}}.
\]
Let $\script{H} = \max_{\eta \in T} H_\eta$. Then for any Weil generator $\alpha \in W$,
\begin{equation}\label{eq:bound-on-alpha-height-in-genus-1-case}
	h(\alpha) \leq \exp\exp\exp\left(2\script{H}\right)^{10^{3^{10}}}.
\end{equation}

\begin{remark}
	One can get a better, but still impractical, upper bound on the height of integral points on genus $1$ curves using the main result in \cite{schmidt1992integer}.
\end{remark}

We can use equation~(\ref{eq:bound-on-alpha-height-in-genus-1-case}) to bound the number $\#W$ of Weil generators in $K$. Because the bound in equation~(\ref{eq:bound-on-alpha-height-in-genus-1-case}) is already impractical, we will not try for an optimal bound. Let $\kappa$ be the number of roots of unity in $K$. Then there are at most $\kappa$ Weil generators which generate the same ideal. This is because if two Weil numbers generate the same ideal, then they differ by a root of unity. Moreover, $\norm_{K/\QQ}(\alpha) = h(\alpha)^{\deg K}$. So the bound on $h(\alpha)$ gives a bound on $\norm_{K/\QQ}(\alpha\sO_K)$. It remains to count the number of ideals in $\sO_K$ of bounded norm. Let $\zeta_K$ denote the Dedekind zeta-function of $K$. If $a_n$ is the number of ideals of $\sO_K$ of norm $n$, then
\[
	\sum_{n \leq M} a_n
	\leq M^2\sum_{n \leq M} \frac{a_n}{n^2}
	\leq M^2\zeta_K(2).
\]
Therefore
\[
	\#W
	\leq
	\kappa\zeta_K(2)\left(\exp\exp\exp\left(2\script{H}\right)^{10^{3^{10}}}\right)^{2\deg K}.
\]

\subsection{Computational Results}

%
%
%
%
We implemented Algorithm~\ref{alg:find-siw-g>=2} in {\tt Sage} to search for Weil generators. We only considered sextic CM fields with class number $1$ because these are the fields required to find super-isolated abelian threefolds (see Section~\ref{sec:super-isolated-variaties}). There are $403$ such fields (see Table~\ref{tbl:results-on-cm-fields-class-number-1}). Many of these fields do not contain any Weil generators. For example, $89$ of these fields do not have a monogenic maximal totally real subfield, see Lemma~\ref{lem:equiv-defn-for-siw}. Our search found Weil generators in $77$ fields. The largest value of $\alpha\overline{\alpha}$ was $83201$. Of the total $644$ Weil generators found, $472$ had the property that $\alpha\overline{\alpha}$ was a prime-power. The largest prime-power value of $\alpha\overline{\alpha}$ was $38461$, which is prime.

\section{Super-Isolated Varieties}\label{sec:super-isolated-variaties}

In this section we are interested in abelian varieties with the following property.

\begin{definition}\label{def:super-isolated-variety}
	Let $q$ be a prime power. We say that an abelian variety $A/\FF_q$ is \emph{super-isolated} if its $\FF_q$-rational isogeny class contains no other $\FF_q$-isomorphism classes.
\end{definition}

The goal of this section is to give examples of super-isolated abelian varieties, as well as to explain their relationship to Weil generators (see Definition~\ref{defn:weil-generator}).

\begin{example}\label{ex:super-isolated-elliptic-curves-over-F2}
	There are $5$ isomorphism classes of elliptic curves over $\FF_2$, and they are given in Table~\ref{tbl:all-curves-over-F2}. Recall that two elliptic curves over a finite field are isogeneous if and only if they share the same number of points. Because each curve in Table~\ref{tbl:all-curves-over-F2} has a different number of points, they lie in distinct isogeny classes. Hence they are all super-isolated.
	\begin{table}
		\centering
		\begin{tabular}{c|c}
			$E$ & $\#E(\FF_2)$ \\\hline
			$y^2 + y = x^3$ &  3 \\
			$y^2 + y = x^3 + x$ &  5 \\
			$y^2 + y = x^3 + x + 1$ & 1 \\
			$y^2 + xy = x^3 + 1$ &  4 \\
			$y^2 + xy + y = x^3 + 1$ & 2
		\end{tabular}
		\caption{Isomorphism classes of elliptic curves over $\FF_2$.}\label{tbl:all-curves-over-F2}
	\end{table}
\end{example}

First we will explain the connection between super-isolated abelian varieties and Weil generators. A theorem of Honda and Tate says that there is a bijection between conjugacy classes of Weil $q$-numbers and isogeny classes of simple abelian varieties over $\FF_q$, see \cite[Sec.~I.6]{waterhouse1971abelian} for references. This bijection works as follows. Let $A/\FF_q$ be an abelian variety that is simple over $\FF_q$, and let $f(x) \in \ZZ[x]$ be the characteristic polynomial of the Frobenius endomorphism of $A$, which has degree $2g$. The Honda-Tate bijection sends the isogeny class of $A$ to the roots of $f$. Let $\pi$ be any root of $f$.

Recall that $A$ is ordinary if $\pi$ is totally imaginary and $(\pi + q/\pi,q) = 1$ \cite[Ch.~7]{waterhouse1969abelian}. In this case, $K = \QQ(\pi)$ is a CM field of degree $2g$ (see \cite[Thm.~7.2]{waterhouse1969abelian} or \cite[Thm.~2]{tate1966endomorphisms}). Theorem~\ref{thm:super-isolated-ordinary-characterization} below shows that if $A$ is ordinary then $A$ is super-isolated if and only if $\pi$ is a Weil generator for $K$ and $K$ has class number $1$. Example~\ref{ex:supersingular-super-isolated-not-weil-generator} below shows that without the ordinary hypothesis, it is possible for $A$ to be super-isolated and $\pi$ to not be a Weil generator.

\begin{remark}
	If $A$ is not ordinary, then $f(x)$ may not be irreducible. For example, by \cite[Thm.~2.9]{maisner2002abelian}, $f(x) = (x^2 - 5)^2$ is the characteristic polynomial of the Frobenius endomorphism of a simple abelian surface over $\FF_5$.
\end{remark}

\begin{theorem}\label{thm:super-isolated-ordinary-characterization}
	Suppose that $A$ is a simple ordinary abelian variety over $\FF_q$. Let $\pi$ denote a root of the characteristic polynomial of the Frobenius endomorphism of $A$ and $K = \QQ(\pi)$. Then $A$ is super-isolated if and only if $\pi$ is a Weil generator for $K$ and $K$ has class number $1$.
\end{theorem}
\begin{proof}
	By \cite[Thm.~7.4]{waterhouse1969abelian}, the set of endomorphism rings that appear in the isogeny class of $A$ are precisely the orders in $K$ containing $\ZZ[\pi,\overline{\pi}]$. So there is one endomorphism ring if and only if $\ZZ[\pi,\overline{\pi}] = \sO_K$. The result then follows from \cite[Thm.~7.2]{waterhouse1969abelian}, which says that the isomorphism classes of abelian varieties in the isogeny class of $A$ whose endomorphism ring is isomorphic to $\sO_K$ form a principal homogeneous space for the class group of $K$. In particular, there is one isomorphism class with endomorphism ring $\sO_K$ if and only if $K$ has class number $1$.
\end{proof}

\begin{example}\label{ex:super-isolated-abelian-surface-F11}
	Let $C$ be the hyperelliptic curve given by $y^2 = x^5 + 4$ over the field $\FF_{11}$. The zeta-function $Z(C,t)$ of $C$ is given by
	\[
		Z(C,t) = \frac{121t^4 + 121t^3 + 51t^2 + 11t + 1}{11t^2 - 12t + 1}.
	\]
	Recall that the reverse of the numerator of $Z(C,t)$ is the characteristic polynomial $f(x)$ of the Frobenius endomorphism of the Jacobian $J$ of $C$ \cite[Ch.~5.2]{cohen2006handbook}. In this case
	\[
		f(x) = x^4 + 11x^3 + 51x^2 + 121x + 121.
	\]
	Because $f$ is irreducible, $J$ is a simple abelian surface. Let $\pi$ be a root of $f(x)$. A straightforward calculation shows that $\pi$ is a Weil generator for $K = \QQ(\pi)$, which has class number $1$ (it is isomorphic to $\QQ(\zeta_5)$). Moreover, $\pi+11/\pi$ is coprime to $11$. Therefore $J$ is super-isolated by Theorem~\ref{thm:super-isolated-ordinary-characterization}.
\end{example}

The following is a straightforward corollary of Theorem~\ref{thm:super-isolated-ordinary-characterization} and Theorem~\ref{thm:W-size}.

\begin{corollary}\label{cor:finite-ord-av-super-isol}
	Let $g \geq 3$. There are finitely many super-isolated simple ordinary abelian varieties of dimension $g$.
\end{corollary}
\begin{proof}
	Let $A$ be a super-isolated simple ordinary abelian variety over a finite field of dimension $g \geq 3$, and let $\pi$ be a root of the characteristic polynomial of the Frobenius endomorphism of $A$. By Theorem~\ref{thm:super-isolated-ordinary-characterization}, $\pi$ is a Weil generator for $K = \QQ(\pi)$, which is a CM field of degree $2g$ and class number $1$. The Honda-Tate theorem says that the map sending the isogeny class of $A$ (which, because $A$ is super-isolated, is equivalent to the isomorphism class of $A$) to the conjugacy class of $\pi$ is injective. Therefore, it is sufficient to count the number of Weil generators in CM fields of degree $2g$ with class number $1$. In \cite{stark1974effective}, Stark proves that the number of such fields is finite, and Theorem~\ref{thm:W-size} says that the number of Weil generators in each such field is finite.
\end{proof}

Corollary~\ref{cor:finite-ord-av-super-isol} suggests that in order to find super-isolated varieties, we should first find CM fields with class number $1$. This is an important problem in number theory. For small values of $g$, all CM fields of degree $2g$ with class number $1$ are known. Table~\ref{tbl:results-on-cm-fields-class-number-1} summarizes the results for $g \leq 3$ and gives references. Many of these fields can also be found in the L-functions and Modular Forms Database \cite{lmfdb}.

\begin{table}
	\centering
	\begin{tabular}{c|c|c}
		Degree & \# of CM fields with class number 1 & Reference\\\hline
		2 & 9 & \cite{gauss1870disquisitiones,stark1967complete,baker1966linear} \\\hline
		4 & \begin{tabular}{@{}c@{}}
			Galois: 54 \\
			Non-Galois: 37
		\end{tabular} & \begin{tabular}{@{}c@{}}
			\cite{setzer1980determination} \\ \cite{louboutin1994determination}
		\end{tabular} \\\hline
		6 & \begin{tabular}{@{}c@{}}
			Galois: 17 \\
			Non-Galois: 386
		\end{tabular} & \begin{tabular}{@{}c@{}}
			\cite{yamamura1994determination,louboutin1992minoration}\\
			\cite{boutteaux2002class,boutteaux2002class2}
		\end{tabular}
	\end{tabular}
	\caption{Summary of the class number 1 problem for CM fields of degree $\leq 6$.}
	\label{tbl:results-on-cm-fields-class-number-1}
\end{table}

\begin{remark}\label{rem:finite-cm-fields-class-number-1}
	It is believed that there is a finite number of CM fields with class number $1$ \cite{stark1974effective}. If this is true, then we do not need to fix $g$ in Corollary~\ref{cor:finite-ord-av-super-isol}.
\end{remark}

\begin{remark}
	If a CM field $K$ has class number $1$, then it avoids some of the obstructions to containing Weil generators mentioned in Section~\ref{sec:weil-gens}. Recall that a necessary condition for $K$ to contain a Weil generator is that $\sO_K$ is a free $\sO_F$-module. This condition was used in Example~\ref{ex:no-relative-generator} to show that $\QQ(\sqrt{60},\sqrt{-2})$ has no Weil generators. However, this condition is always satisfied when $K$ has class number $1$ because in that case $F$ also has class number $1$ \cite[Prop.~4.11]{washington1997introduction}. Hence $\sO_F$ is a PID, so by the structure theorem for modules over a PID, $\sO_K$ is a free $\sO_F$-module.
\end{remark}

\subsection{Examples of Super-Isolated Elliptic Curves}

In this section we will give examples of super-isolated elliptic curves, as well as extend Theorem~\ref{thm:super-isolated-ordinary-characterization} to include the case of supersingular curves.

In \cite{schoof1987nonsingular}, Schoof gave a formula for the size of the isogeny class of an elliptic curve $E/\FF_q$ in terms of $\#E(\FF_q)$ and $q$. The following is a straightforward consequence of that formula.

\begin{proposition}\label{prop:super-isolated-curve-characterization}
	Let $q = p^a$ be a prime power, $E$ be an elliptic curve over $\FF_q$, and $t = q + 1 - \#E(\FF_q)$. Then $E/\FF_q$ is super-isolated if and only if one of the following holds:
	\begin{enumerate}
		\item\label{ex:superiso-curve-case-1} $t \not\equiv 0 \mod{p}$ and $t^2 - 4q \in \{-3,-4,-7,-8,-11,-19,-43,-67,-163\}$.
		\item\label{ex:superiso-curve-case-2} $p \in \{2,3,5,7,13\}$, and $t^2 = 4q$.
		\item\label{ex:superiso-curve-case-3} $p \in \{2,3\}$, and $t^2 = p^{a+1}$.
		\item\label{ex:superiso-curve-case-4} $p = 2$ and $t = 0$.
		\item\label{ex:superiso-curve-case-5} $p = 3$ and $t^2 = q$.
	\end{enumerate}
\end{proposition}
\begin{proof}
	Case~\ref{ex:superiso-curve-case-1} corresponds to ordinary elliptic curves and follows from Theorem~\ref{thm:super-isolated-ordinary-characterization}. The set of values for $t^2 - 4q$ are the discriminants of the imaginary quadratic fields with class number $1$, see Table~\ref{tbl:results-on-cm-fields-class-number-1} for references. The rest of the cases follow directly from a special case of \cite[Thm.~4.6]{schoof1987nonsingular}.
\end{proof}

Table~\ref{tbl:examples-super-isolated-curves} gives examples of curves satisfying each case in Proposition~\ref{prop:super-isolated-curve-characterization}.

\begin{table}
	\centering
	\begin{tabular}{c|c|c|c}
		Case & $E$ & $q$ & $t$ \\\hline
		\ref{ex:superiso-curve-case-1} & $y^2 = x^3 + 4$ & $7$ & $5$ \\
		\ref{ex:superiso-curve-case-2} & $y^2 + y = x^3$ & $4$ & $-4$ \\
		\ref{ex:superiso-curve-case-3} & $y^2 + y = x^3 + x + 1$ & $2$ & $2$ \\
		\ref{ex:superiso-curve-case-4} & $y^2 + y = x^3$ & $2$ & $0$ \\
		\ref{ex:superiso-curve-case-5} & $y^2 = x^3 - x + \beta$ & $9$ & $3$
	\end{tabular}
	\caption{Examples of super-isolated elliptic curves. Here $\beta \in \FF_{9}$ satisfies $x^2 - x - 1$.}
	\label{tbl:examples-super-isolated-curves}
\end{table}

Recall that the property of being super-isolated depends on the base field (see Definition~\ref{def:super-isolated-variety}). It is often the case that a variety may be super-isolated over $\FF_q$, but not over an extension of $\FF_q$.

\begin{example}
	Let $E/\FF_7$ be the elliptic curve defined by $y^2 = x^3 + 4$. Note that $\#E(\FF_7) = 3$, so $E$ is super-isolated by Proposition~\ref{prop:super-isolated-curve-characterization}. However, over the extension field $\FF_{49}$, $E$ is isogeneous, but not isomorphic, to the curve given by $y^2 = x^3 + 6\beta x + 4$, where $\beta \in \FF_{49}$ is a root of $x^2 + 6x + 3$. Therefore $E/\FF_7$ is super-isolated but $E/\FF_{49}$ is not.
\end{example}

It is possible for an abelian variety that is not super-isolated over the base field to become super-isolated over an extension field. For elliptic curves, this phenomenon can only occur for supersingular curves. To see why, let $E/\FF_q$ be an ordinary elliptic curve and suppose that $E/\FF_{q^k}$ is super-isolated for some extension field $\FF_{q^k}$. Then any curve $E'/\FF_q$ which is isogeneous to $E/\FF_q$ must become isomorphic to $E$ over $\FF_{q^k}$. This means that $E'/\FF_q$ is a twist of $E/\FF_q$ (see \cite[Ch.~X.5]{silverman2009arithmetic}). One can show that for ordinary curves, non-trivial twists are never isogeneous (here a non-trivial twist is one that is not isomorphic over the base field). For example, if $E'/\FF_q$ is a quadratic twist of $E/\FF_q$, and $t = p + 1 - \#E(\FF_q)$, then $\#E'(\FF_q) = p + 1 + t$. Since $E$ is ordinary, $t \neq 0$ so $E'$ lies in a different isogeny class (see also \cite[Pg.~542]{waterhouse1969abelian}).

\begin{example}
	Let $E/\FF_5$ be the supersingular elliptic curve given by $y^2 = x^3 + 2$. Then $E$ is isogeneous, but not isomorphic, to the curve $E'/\FF_5$ given by $y^2 = x^3 + 1$. However, by applying Proposition~\ref{prop:super-isolated-curve-characterization}, one can check that $E/\FF_{25}$ is super-isolated. In this case, $E/\FF_{5}$ and $E'/\FF_{5}$ become isomorphic over $\FF_{25}$.
\end{example}

Another possibility is that a super-isolated variety could stay super-isolated in every extension.

\begin{example}
	Let $E/\FF_2$ be the curve given by $y^2 + y = x^3$. In this case, one can compute (see \cite[Exercise~5.13]{silverman2009arithmetic}) that
	\[
		t(E/\FF_{2^a}) =
		\begin{cases}
		0 & a \text{ odd} \\
		2(-2)^{a/2} & a \text{ even}.
		\end{cases}
	\]
	By Proposition~\ref{prop:super-isolated-curve-characterization}, this shows that $E/\FF_{2^a}$ is super-isolated for every $a \geq 1$. This example is somewhat exceptional because $E$ is supersingular.
\end{example}

The following example shows that if $E$ is supersingular, then it is possible that $E$ is super-isolated but the Frobenius endomorphism does not correspond to a Weil generator.

\begin{example}\label{ex:supersingular-super-isolated-not-weil-generator}
	Let $E/\FF_9$ be the supersingular elliptic curve in the last row of Table~\ref{tbl:examples-super-isolated-curves}. The characteristic polynomial $f(x)$ of the Frobenius endomorphism of $E$ is $f(x) = x^2 - 3x + 9$. Let $\pi$ be a root of $f(x)$ and $K = \QQ(\pi)$. The discriminant of $f(x)$ is $-27$, so $\ZZ[\pi]$ has index $3$ in $\sO_K$. In particular, $\pi$ is not a Weil generator for $K$. This shows that the ordinary hypothesis in Theorem~\ref{thm:super-isolated-ordinary-characterization} is necessary.
\end{example}

One way to construct super-isolated elliptic curves over large prime fields is to use the \emph{complex-multiplication (CM) method}. A detailed summary of the CM method can be found in \cite[Ch.~18]{cohen2006handbook}. Essentially, using the CM method to generate super-isolated curves works as follows:
\begin{enumerate}
	\item Choose a quadratic imaginary field $K$ with class number $1$.
	\item Find an elliptic curve $E/\CC$ whose endomorphism ring is isomorphic to $\sO_K$.
	\item Choose a Weil generator $\pi$ (with non-zero trace) for $K$ such that $p = \pi\overline{\pi}$ is prime and is $\geq 5$. This will ensure that the resulting curve is ordinary.
	\item Find a twist of the reduction $E/\FF_p$ whose Frobenius endomorphism corresponds to $\pi$.
\end{enumerate}

If also the only roots of unity in $K$ are $\pm 1$, then we can always use $E$ as opposed to one of its twists. To see why, recall that in this case there is only a single twist of $E$: the quadratic twist. If $E$ has $p + 1 - \tr_{K/\QQ}(\pi)$ points, then the quadratic twist of $E$ has $p + 1 + \tr_{K/\QQ}(\pi)$. It follows from Proposition~\ref{prop:super-isolated-curve-characterization} that if one of them is isolated, then they both are.

\begin{example}\label{ex:cm-method-for-curves}
	Let $K = \QQ(\sqrt{-2})$. The endomorphism ring (over $\CC$) of the elliptic curve $E/\QQ$ given by $y^2 = x^3 - x^2 - 3x - 1$ is isomorphic to $\sO_K$. Recall that every Weil generator for $K$ is of the form $\pi = b \pm \sqrt{-2}$ for some integer $b$. Thus searching for primes $p$ such that $E/\FF_p$ is super-isolated reduces to finding values of $b$ such that $p = b^2 + 2$ is prime.
\end{example}

The CM method described above can also be used to generate curves with useful properties, such as prime order and a base field with low hamming weight.
\begin{example}
	Let $\pi =  2^{127} + 2^{25} + 2^{12} + 2^6 + (1 - \sqrt{-3})/2$. Then $p = \pi\overline{\pi}$ is a $255$ bit prime with Hamming weight $14$. Using the CM method, we found a super-isolated curve $E/\FF_p$ given by $y^2 = x^3 + 19$ that has $\#E(\FF_p) = (\pi - 1)(\overline{\pi} - 1)$ points, which is also a $255$ bit prime.
\end{example}

\subsection{Examples in Higher Dimensions}

In this section we will give examples of super-isolated abelian varieties in dimension $g \geq 2$. One way to construct these varieties is to fix a prime $p$ and randomly choose curves over $\FF_p$ until the Jacobian is super-isolated. We can check if the Jacobian $J$ of a curve $C/\FF_p$ is super-isolated using the zeta-function of $C$ as in Example~\ref{ex:super-isolated-abelian-surface-F11}. Some examples of curves found this way are given in Table~\ref{tbl:examples-of-super-isolated-jacobians}. Because super-isolated abelian varieties are rare, this method of search is impractical when $p$ is large.

\begin{table}
	\centering
	\begin{tabular}{c|c|c}
		Curve & Genus & Field \\\hline
		$y^2 = x^3 + 5x^2 + x + 1$ & $1$ & $\FF_7$ \\
		$y^2 = x^5 + 2x^4 + 4x^3 + x^2 + x + 4$ & $2$ & $\FF_5$ \\
		$y^2 = x^7 + x^5 + x + 2$ & $3$ & $\FF_3$ \\
		$y^2 + (x^5 + x^3 + 1)y = x^9 + x^6$ & $4$ & $\FF_2$
	\end{tabular}
	\caption{Examples of hyperelliptic curves with super-isolated Jacobians.}
	\label{tbl:examples-of-super-isolated-jacobians}
\end{table}

Another way to construct super-isolated abelian varieties is described in Example~\ref{ex:cm-method-for-surfaces}. This method is based on a generalization of the CM method to higher dimensions, see \cite[Ch.~18]{cohen2006handbook}.

\begin{example}\label{ex:cm-method-for-surfaces}
	Let $\pi \in K = \QQ(\zeta_5)$ be a totally imaginary Weil $p$-number such that $p$ splits in $K$. Let $C/\QQ$ be the curve defined by $y^2 = x^5 - 1$. The endomorphism ring over $\CC$ of the Jacobian $J$ of $C$ is isomorphic to $\sO_K$. This means that there is a twist $C'/\FF_p$ of $C/\FF_p$ such that the Frobenius endomorphism of the Jacobian of $C'$ satisfies the minimal polynomial of $\pi$. If $\pi$ is also a Weil generator for $K$, then the resulting surface will be super-isolated. For example, $\pi = 45\zeta_5^3 - 10\zeta_5^2 + 34\zeta_5 - 2320$ is a Weil generator for $K$ with $p = \pi\overline{\pi} = 5465351$ prime. In this case, the Jacobian of the curve $y^2 = x^5 - 4$ over $\FF_p$ is super-isolated.
\end{example}

The CM method is difficult for genus $g > 5$ or fields $K$ of large discriminant or degree. The main difficulty is writing down an appropriate global variety $A/\CC$. See \cite{balakrishnan2016constructing} for a construction in dimension $3$. Moreover, for $g \geq 3$, Theorem~\ref{thm:W-size} suggests that there are few super-isolated abelian varieties of dimension $g$.

It is sometimes possible to use properties of the field $K$ to construct slightly larger examples than we could find by randomly searching through curves.

\begin{example}\label{ex:larger-genus-3-example}
	Suppose that $K$ is a sextic CM field with class number $1$ that contains $\sqrt{-1}$, and suppose that $C/\FF_p$ is a curve of genus $3$ whose Jacobian has an endomorphism ring isomorphic to $\sO_K$. Then it follows from \cite[Cor.~18.17]{cohen2006handbook} that $C$ is isomorphic to a curve of the form $y^2 = x^7 + x^5 + ax^3 + bx$ for some $a,b \in \FF_p$. This means that we can search for a super-isolated abelian threefold by first finding a Weil generator $\pi$ for $K$ such that $p = \pi\overline{\pi}$ is prime. Then we can range over all pairs $(a,b) \in (\FF_p)^2$ and check if the Jacobian of the resulting curve is super-isolated. We used this method to find the curve $y^2 = x^7 + x^5 + 160x^3 + 79x$ over the field $\FF_{353}$. The characteristic polynomial $f(x)$ of the Frobenius endomorphism of the Jacobian $J$ of this curve is
	\[
		x^6 - 88x^5 + 3440x^4 - 80400x^3 + 1214320x^2 - 10965592x + 43986977.
	\]
	The roots of this polynomial are Weil generators for the field generated by the root $\pi$ of $f$ (which is a sextic CM field with class number $1$); hence $J$ is a super-isolated abelian threefold.
\end{example}

\bibliographystyle{abbrv}
\bibliography{references}

\end{document}